\newif\ifaistats
\aistatsfalse

\ifaistats
\documentclass[twoside]{article}

\usepackage{natbib}

\usepackage[accepted]{aistats2021}
\else
\documentclass{article}
\fi
\usepackage{amsmath}
\ifaistats
\usepackage{balance}
\else
\usepackage{times}
\fi
\usepackage{slashbox}
\usepackage{amssymb}
\usepackage{amsthm}
\usepackage{thmtools}
\usepackage{thm-restate}
\usepackage{bm}
\usepackage{graphicx}
\usepackage{caption}
\ifaistats
\else
\usepackage[margin = 1in]{geometry}
\usepackage[sort,comma,authoryear,round]{natbib}
\fi
\usepackage{color}
\usepackage{subfigure}
\usepackage{comment}
\usepackage{enumitem}
\usepackage[colorlinks,urlcolor=blue,citecolor=blue,linkcolor=blue]{hyperref}
\usepackage[colorinlistoftodos]{todonotes}

\newcommand{\innp}[1]{\left\langle #1 \right\rangle}

\newcommand{\zeros}{\textbf{0}}
\newcommand{\vx}{\mathbf{x}}
\newcommand{\tF}{\Tilde{F}}

\newcommand{\cu}{\mathcal{U}}
\newcommand{\cf}{\mathcal{F}}
\newcommand{\cfb}{\bar{\mathcal{F}}}

\newcommand{\vy}{\mathbf{y}}
\newcommand{\vz}{\mathbf{z}}
\newcommand{\vv}{\mathbf{v}}

\newcommand{\vw}{\mathbf{w}}

\newcommand{\vg}{\mathbf{g}}
\newcommand{\vu}{\mathbf{u}}
\newcommand{\vub}{\bar{\mathbf{u}}}

\newcommand{\veta}{\bm{\eta}}
\newcommand{\vetab}{\bar{\bm{\eta}}}

\newcommand{\defeq}{\stackrel{\mathrm{\scriptscriptstyle def}}{=}}

\newcommand{\rr}{\mathbb{R}}
\newcommand{\ee}{\mathbb{E}}

\graphicspath{{imgs/}}

\makeatletter
\def\mathcolor#1#{\@mathcolor{#1}}
\def\@mathcolor#1#2#3{%
  \protect\leavevmode
  \begingroup
    \color#1{#2}#3%
  \endgroup
}
\makeatother

\newcommand*{\vsepfbox}[1]{%
  \begingroup
    \sbox0{\fbox{#1}}%
    \setlength{\fboxrule}{0pt}%
    \mbox{\kern-\fboxsep\fbox{\unhbox0}\kern-\fboxsep}%
  \endgroup
}

\theoremstyle{plain} \numberwithin{equation}{section}
\newtheorem{theorem}{Theorem}[section]
\numberwithin{theorem}{section}

\newtheorem{lemma}[theorem]{Lemma}

\newtheorem{fact}[theorem]{Fact}
\theoremstyle{definition}
\newtheorem{definition}[theorem]{Definition}

\newtheorem{remark}[theorem]{Remark}

\newtheorem{assumption}{Assumption}

\DeclareMathOperator*{\argmin}{argmin}

\DeclareMathOperator*{\argsup}{argsup}

\makeatletter
\newcommand{\subalign}[1]{%
  \vcenter{%
    \Let@ \restore@math@cr \default@tag
    \baselineskip\fontdimen10 \scriptfont\tw@
    \advance\baselineskip\fontdimen12 \scriptfont\tw@
    \lineskip\thr@@\fontdimen8 \scriptfont\thr@@
    \lineskiplimit\lineskip
    \ialign{\hfil$\m@th\scriptstyle##$&$\m@th\scriptstyle{}##$\hfil\crcr
      #1\crcr
    }%
  }%
}
\makeatother

\newcommand{\eg}{\textsc{eg}}
\newcommand{\egp}{\textsc{eg}$+$}
\newcommand{\egpp}{\textsc{eg}$_p+$}

\renewcommand{\cite}[1]{\citep{#1}}

\ifaistats

\fi

\ifaistats
\else
\title{Efficient Methods for Structured\\  Nonconvex-Nonconcave Min-Max Optimization}
\author{Jelena Diakonikolas\\
University of Wisconsin-Madison\\
\texttt{jelena@cs.wisc.edu}
\and
Constantinos Daskalakis\\
Massachusetts Institute of Technology\\
\texttt{costis@csail.mit.edu}
\and
Michael I.~Jordan\\
UC Berkeley\\
\texttt{jordan@cs.berkeley.edu}}
\date{}
\fi

\begin{document}

\ifaistats
\else
\maketitle
\fi
\begin{abstract}
    The use of min-max optimization
    in the adversarial training of deep neural network classifiers, and the training of generative adversarial networks has motivated the study of nonconvex-nonconcave optimization objectives, which frequently arise in these applications. Unfortunately, recent results have established that even approximate first-order stationary points of such objectives 
    are intractable, even under smoothness conditions, 
    motivating the study of min-max objectives with additional  structure. We introduce a new class of structured nonconvex-nonconcave min-max optimization problems, proposing a  generalization of the extragradient algorithm which provably converges to a stationary point. The algorithm applies not only to Euclidean spaces, but also to general $\ell_p$-normed finite-dimensional real vector spaces. We also discuss its stability under stochastic oracles and provide bounds on its sample complexity. Our iteration complexity and sample complexity bounds either match or improve the best known bounds for the same or less general nonconvex-nonconcave settings, such as those that satisfy variational coherence or in which a weak solution to the associated variational inequality problem is assumed to exist.
\end{abstract}

\section{Introduction}

Min-max optimization and min-max duality theory lie at the foundations of game theory and mathematical programming, and have found far-reaching applications across a range of disciplines, including complexity theory, statistics, control theory, and online learning theory. Most recently, min-max optimization has played an important role in machine learning, notably in the adversarial training of deep neural network classifiers and the training of generative deep neural network models. These recent applications have heightened the importance of solving min-max optimization problems with nonconvex-nonconcave objectives, taking the following general form:
\begin{align}\label{eq:min-max}
    \min_{\vx} \max_{\vy} f(\vx,\vy),
\end{align}
where $\vx$ and $\vy$ are real-valued vectors and $f$ is not (necessarily) convex in $\vx$ for all $\vy$ and/or not (necessarily) concave in $\vy$ for all $\vx$.  There may also be  constraints on $\vx$ and $\vy$, and in many applications $\vx$ and $\vy$ are high-dimensional vectors.

When the objective function is not convex-concave, von Neumann's celebrated min-max theorem fails to apply, and so do most standard optimization 
methods for  solving~\eqref{eq:min-max}. This has motivated several lines of investigation, which include extensions of the min-max theorem beyond convex-concave objectives (e.g.~Sion's theorem for quasiconvex-quasiconcave objectives), and the pursuit of computational procedures that target solutions to~\eqref{eq:min-max} even in the absence of a min-max theorem; see Section~\ref{sec:related work} for a review of recent work. Of course, without strong assumptions on~$f$, \eqref{eq:min-max} is an intractable problem,  at least as intractable as general nonconvex optimization. Thus, the literature has targeted locally optimal solutions, in the same spirit as the  targeting of local optima in non-convex optimization. Naturally, there are various notions of local optimality that have been studied in the literature.  Our focus here will be on the simplest such notion, namely first-order local optimality, for which, despite the apparent simplicity, many challenges arise~\cite{daskalakis2018limit,mazumdar2018convergence}. 

In contrast to classical optimization problems, where useful results can be obtained with very mild assumptions on the objective function, in min-max optimization it is necessary to impose non-trivial assumptions on $f$, even when the goal is only to compute locally optimal solutions. Indeed,~\citet{daskalakis2020complexity} establish intractability results in the constrained setting of the problem, wherein first-order locally optimal solutions are guaranteed to exist whenever the objective is smooth. Moreover, they show that even the computation of \emph{approximate} solutions is {\tt PPAD}-complete and, if the objective function is accessible through value-queries and gradient-queries,  exponentially many such queries are necessary (in particular, exponential in at least one of the following: the inverse approximation parameter, the smoothness constant of $f$, or the diameter of the constraint set). 

We expect similar intractability results to hold in the unconstrained case, which is the case considered in this paper, even when restricting to  smooth objectives that have a non-empty set of optimal solutions.\footnote{Note that these are stationary points of $f$ in this case.} Indeed, fixed-point complexity-based intractability results for the constrained case are typically extendable to the unconstrained case, by embedding the hard instances within an unbounded domain. 

Relatedly, we already know that the unconstrained Stampacchia variational inequality (SVI) problem for Lipschitz continuous operators $F:\rr^d \to \rr^d$---a problem which includes the unconstrained case of~\eqref{eq:min-max} by setting $F([\subalign{\vx\\ \vy}])=\big[\subalign{\nabla_{\vx}f(\vx, \vy)\\ - \nabla_{\vy}f(\vx, \vy)}\big]$---is computationally intractable, even when restricting to operators that have a non-empty set of SVI solutions.\footnote{We formally define the \emph{Stampacchia variational inequality} problem, \eqref{eq:SVI}, in Section~\ref{sec:prelims}.  We also define the harder \emph{Minty variational inequality} problem, \eqref{eq:MVI}, in the same section.} This is because: (i) $F$ is Lipschitz-continuous if and only if the operator $T(\vu) = \vu - F(\vu)$ is Lipschitz-continuous; (ii) for~$\epsilon \geq 0$, points $\vub \in \rr^d$ such that $\|F(\vub)\|_2 \leq \epsilon$ satisfy  $\|T(\vub) - \vub\|_2\leq \epsilon$, i.e.~they are
$\epsilon$-approximate fixed points of~$T$, and vice versa; and (iii) 
it is known that finding approximate fixed points of Lipschitz operators over $\mathbb{R}^d$ is {\tt PPAD}-hard, even when the operators are guaranteed to have fixed points~\cite{papadimitriou1994complexity}. Moreover, if we restrict  attention to algorithms that only make value queries to $T$ (i.e.~$F$, which corresponds to the type of access that all first-order algorithms have), the query complexity becomes exponential in the dimension~\cite{hirsch1989exponential}. Finally, by the equivalence of norms, these results extend to arbitrary $\ell_p$-normed finite dimensional real vector spaces. Of course, for these intractability results for SVI  to apply to the  nonconvex-nonconcave min-max problem~\eqref{eq:min-max}, one would need to prove that these complexity results extend to operators $F$ constructed from a smooth function $f$ by setting $F([\subalign{\vx\\ \vy}])=\big[\subalign{\nabla_{\vx}f(\vx, \vy)\\ - \nabla_{\vy}f(\vx, \vy)}\big]$.


\paragraph{Our contributions.} Given the aforedescribed intractability results, our goal is to identify structural properties that make it possible to solve min-max optimization problems with smooth objectives. Focusing on the unconstrained setting of~\eqref{eq:min-max}, we view it as a special case (obtained by considering the operator $F([\subalign{\vx\\ \vy}])=\big[\subalign{\nabla_{\vx}f(\vx, \vy)\\ - \nabla_{\vy}f(\vx, \vy)}\big]$) of the unconstrained variational inequality problem~\eqref{eq:SVI}, and consider instead this more general problem. We identify conditions for $F$ under which a generalized version of the extragradient method of~\citet{korpelevich1976extragradient}, which we propose, converges to a solution of~\eqref{eq:SVI} (or, in the special case of~\eqref{eq:min-max}, to a stationary point of~$f$) at a rate of $1/\sqrt{k}$ in the number of iterations $k$. Our condition, presented as Assumption~\ref{assmpt:cohypo}, postulates that there exists a solution to~\eqref{eq:SVI} that only violates the stronger~\eqref{eq:MVI} requirement in a controlled manner that we delineate. Our generalized extragradient method is based on an aggressive interpolation step, as specified by~\eqref{eq:mod-eg}, and our main convergence result is Theorem~\ref{thm:convergence-of-eg+}.  We additionally show, in Theorems~\ref{thm:egpp_deterministic} and~\ref{thm:egpp-stochastic}, that the algorithm converges in non-Euclidean settings, under {the stronger condition} that an~\eqref{eq:MVI} solution exists, or when we only have stochastic oracle access to $F$ (or, in the special case of~\eqref{eq:min-max}, to the gradient of $f$).

The condition on $F$ under which our main result applies is weaker than the assumption that a solution to~\eqref{eq:MVI} exists~\cite{zhou2017stochastic, mertikopoulos2018optimistic, Malitsky2019, song2020optimistic}, an assumption which is already satisfied by several interesting families of min-max objectives, including quasiconvex-concave families or starconvex-concave families. Our significantly weaker condition applies in particular to (min-max objectives $f$ with corresponding) operators $F$ that are negatively comonotone~\cite{bauschke2020generalized} or positively cohypomonotone~\cite{combettes2004proximal}. These conditions have been studied in the literature for at least a couple of decades, but only asymptotic convergence results were available prior to our work for computing solutions to~\eqref{eq:SVI}. In contrast, our rates are asymptotically identical to the rates that we would get under the stronger assumption that a solution to~\eqref{eq:MVI} exists, and sidestep the intractability results for~\eqref{eq:min-max} suggested by~\citet{daskalakis2020complexity} for general smooth objectives.


\subsection{Further Related Work} \label{sec:related work}

\begin{table*}[t!]
\centering
\caption{Comparison of iteration complexities required to find a point $\vx$ with $\|F(\vx)\|_{p^*}\leq \epsilon$ using deterministic algorithms, where $\epsilon >0,$ $F:\rr^d \to\rr^d$ is a Lipschitz operator satisfying Assumption~\ref{assmpt:cohypo} (Section~\ref{sec:prelims}) with $\rho \geq 0$. Parameter $p$ determines the $\ell_p$ setup, and $p^* = \frac{p}{p-1}$ is the exponent conjugate to $p.$ Only the dependence on $\epsilon$ and possibly the dimension $d$ is shown; the dependence on other problem parameters is comparable for all the results. $\widetilde{O}$ hides logarithmic factors. `---' indicates that the result does not exist/is not known.}
\begin{tabular}{|l||*{5}{c|}}
\hline \label{table:det-results}
\backslashbox{\bf Paper}{\bf Setup} & $\rho \in (0, \frac{1}{4L}),$ $p = 2$ & $\rho = 0,$ $p = 2$ & $\rho = 0,$ $p \in (1, 2)$ & $\rho = 0,$ $p > 2$\\
\hline
\cite{dang2015convergence} & --- & $O(\frac{1}{\epsilon^2})$ & $O(\frac{\mathrm{poly}(d^{{1}/{p}-{1}/{2}})}{\epsilon^2})$ & $O(\frac{\mathrm{poly}(d^{{1}/{2}-{1}/{p}})}{\epsilon^2})$\\
\hline
\cite{lin2018solving} & --- & $\widetilde{O}(\frac{1}{\epsilon^2})$ & --- & ---\\
\hline
\cite{song2020optimistic} & --- & $O(\frac{1}{\epsilon^2})$ & $O(\frac{1}{\epsilon^2})$ & --- \\
\hline
\textbf{This Paper} & $O(\frac{1}{\epsilon^2})$ & $O(\frac{1}{\epsilon^2})$ & $O(\frac{1}{\epsilon^2})$ & $O(\frac{1}{\epsilon^p})$\\
\hline
\end{tabular}
\end{table*}

A large number of recent works target identifying practical first-order, low-order, or efficient online learning methods for solving min-max optimization problems in a variety of settings, ranging from the well-behaved setting of convex-concave objectives to the challenging setting of nonconvex-nonconcave objectives. There has been substantial work for convex-concave and nonconvex-concave objectives, targeting the computation of min-max solutions to~\eqref{eq:min-max} or, respectively, stationary points of $f$ or $\Phi(\vx):=\max_{\vy} f(\vx,\vy)$. This work has focused on attaining improved convergence rates~\cite{kong2019accelerated, lin2019gradient, thekumparampil2019efficient, nouiehed2019solving, lu2019hybrid, zhao2019optimal, alkousa2019accelerated, AzizianMLG20, GolowichPDO, lin2020near, diakonikolas2020halpern} and/or obtaining last-iterate convergence guarantees~\cite{daskalakis2017training, daskalakis2018limit, mazumdar2018convergence, MertikopoulosPP18, lin2018solving, hamedani2018primal, adolphs2018local, daskalakis2019last, liang2019interaction, gidel2019negative, mokhtari2019unified, abernethy2019last,liu2019towards}. 

In the nonconvex-nonconcave setting, research has focused on identifying different
notions of local min-max solutions \cite{daskalakis2018limit,mazumdar2018convergence,jin2019minmax, MangoubiV20} and studying 
the existence and (local) convergence properties of learning methods to these
points \cite{wang2019solving, MangoubiSV20, MangoubiV20}. As already discussed, recent work of~\citet{daskalakis2020complexity} shows that, for general smooth objectives, the computation of even approximate first-order locally optimal min-max solutions is intractable, motivating the identification of structural assumptions on the objective function for which these intractability barriers can be bypassed. 

An example such assumption, which is closely related to the one made in this work, is that an~\eqref{eq:MVI} solution exists for the operator $F([\subalign{\vx\\ \vy}])=\big[\subalign{\nabla_{\vx}f(\vx, \vy)\\ - \nabla_{\vy}f(\vx, \vy)}\big]$, as studied by~\citet{zhou2017stochastic, lin2018solving, mertikopoulos2018optimistic, Malitsky2019, liu2019towards, song2020optimistic}. As we have already discussed, the assumption we make for our main result in this work is weaker. Table~\ref{table:det-results} provides a comparison of our results to those of existing works, considering the deterministic setting (i.e.~having exact value access to~$F$). 

In unconstrained Euclidean setups, the best known convergence rates are of the order $1/\sqrt{k}$~\cite{dang2015convergence,song2020optimistic}, under the assumption that an~\eqref{eq:MVI} solution exists. We obtain the same rate under our weaker Assumption~\ref{assmpt:cohypo}. Moreover, under our weaker assumption, we show that the accumulation points of the sequence of iterates of our algorithm are \eqref{eq:SVI} solutions. This was previously established for alternative algorithms and under the stronger assumption that an \eqref{eq:MVI} solution exists~\cite{mertikopoulos2018optimistic,Malitsky2019}. 

When it comes to more general $\ell_p$ norms, \citet{mertikopoulos2018optimistic} establish the asymptotic convergence of the iterates of an optimistic variant of the mirror descent algorithm, under the assumption that an~\eqref{eq:MVI} solution exists, but they do not provide any convergence rates. On the other hand, \citet{dang2015convergence} prove a $1/\sqrt{k}$ rate of convergence for a variant of the mirror-prox algorithm in general normed spaces. This result, however, 
requires the regularizing (prox) function to be both smooth and strongly convex w.r.t.~the same norm, and the constant in the convergence bound scales at least linearly with the condition number of the prox function. It is well-known that no function can be simultaneously smooth and strongly convex w.r.t.~an $\ell_p$ norm with $p\neq 2$ and have a condition number independent of the dimension~\cite{borwein2009uniformly}. In fact, unless $p$ is trivially close to $2,$ we only know of functions whose condition number would scale polynomially with the dimension.

Very recent (and independent) work of~\citet{song2020optimistic} proposes an optimistic dual extrapolation method with linear convergence for a class of problems that have a ``strong'' \eqref{eq:MVI} solution. (In particular, their assumption is that there exists $\vu^* \in \rr^d$ such that $\forall \vu \in \rr^d$: $\innp{F(\vu), \vu - \vu^*} \geq m\|\vu - \vu^*\|^2$ for some constant $m \geq 0$; the case $m=0$ recovers the existence of a standard \eqref{eq:MVI} solution.) Their result only applies to norms that are strongly convex, which in the case of $\ell_p$ norms is true only for $p \in (1, 2].$ In that case, our results match those of~\citet{song2020optimistic}. 
For the case of stochastic oracle access to $F,$ our bounds also match those of \citet{song2020optimistic} for $p \in (1, 2],$ and we also handle the case $p >2$ which is not covered by \citet{song2020optimistic}.

Finally, it is worth noting that~\citet{zhou2017stochastic, mertikopoulos2018optimistic, Malitsky2019, song2020optimistic} consider constrained optimization setups, which are not considered in our work. We believe that generalizing our results to constrained setups is possible, and defer such generalizations to future work.

\section{Notation and Preliminaries}\label{sec:prelims}

We consider real $d$-dimensional spaces $(\rr^d, \|\cdot\|_p),$ where $\|\cdot\|_p$ is the standard $\ell_p$ norm for $p \geq 1.$ In particular, $\|\cdot\|_2 = \sqrt{\innp{\cdot, \cdot}}$ is the $\ell_2$ (Euclidean) norm and $\innp{\cdot, \cdot}$ denotes the inner product. When the context is clear, we omit the subscript 2 and just write $\|\cdot\|$ for the Euclidean norm $\|\cdot\|_2.$ Moreover, we denote by $p^* = \frac{p}{p-1}$ the exponent conjugate to $p.$

We are interested in finding stationary points for min-max problems of the form:
\begin{equation}\tag{P}\label{problem:min-max}
    \min_{\vx \in \rr^{d_1}}\max_{\vy \in \rr^{d_2}} f(\vx, \vy),
\end{equation}
where $f$ is a smooth (possibly nonconvex-nonconcave) function and $d_1 + d_2 = d$. In this case, stationary points can be defined as the points at which the gradient of $f$ is the zero vector. As is standard, the $\epsilon$-approximate variant of this problem for $\epsilon > 0$ is to find a point $(\vx, \vy)\in \rr^{d_1}\times \rr^{d_2}$ such that $\|\nabla f(\vx, \vy)\|_{p^*} \leq \epsilon.$ 


We will study Problem~\eqref{problem:min-max} through the lens of variational inequalities, described in Section~\ref{sec:VIs}. To do so, we consider the operator $F: \rr^d \to \rr^d$ defined via $F(\vu) = \big[\subalign{\nabla_{\vx}f(\vx, \vy)\\ - \nabla_{\vy}f(\vx, \vy)}\big],$ where $\vu = \big[\subalign{\vx \\ \vy}\big]$ and where $\nabla_{\vx} f$ (respectively, $\nabla_{\vy} f$) denotes the gradient of $f$ w.r.t.~$\vx$ (respectively, $\vy$). It is clear that $F$ is Lipschitz-continuous whenever $f$ is smooth and that $\|F(\vu)\|_{p^*}\leq \epsilon$ for $\vu = \big[\subalign{\vx \\ \vy}\big]$ holds if and only if $\|\nabla f(\vx, \vy)\|_{p^*} \leq \epsilon$.

%
\subsection{Variational Inequalities and Structured (Possibly Non-Monotone) Operators} \label{sec:VIs}
Let $F: \rr^d \rightarrow \rr^d$ be an operator that is $L$-Lipschitz-continuous w.r.t.~$\|\cdot\|_p$:
\begin{equation}\label{eq:lipschitzness}\tag{\textsc{a}$_1$}
    (\forall \vu, \vv \in \rr^d):\quad \|F(\vu) - F(\vv)\|_{p^*} \leq L\|\vu - \vv\|_p. 
\end{equation}
$F$ is said to be \emph{monotone} if:
\begin{equation}\label{eq:mon-op}
    (\forall \vu, \vv \in \rr^d):\quad \innp{F(\vu) - F(\vv), \vu - \vv} \geq 0.
\end{equation}

Given a closed convex set $\cu\subseteq \rr^d$ and an operator $F,$ the \emph{Stampacchia Variational Inequality} problem consists in finding $\vu^* \in \rr^d$ such that:
\begin{equation}\label{eq:SVI}\tag{\textsc{svi}}
    (\forall \vu \in {\cu}): \quad \innp{F(\vu^*), \vu - \vu^*} \geq 0.
\end{equation}
In this case, $\vu^*$ is referred to as the \emph{strong solution} to the variational inequality corresponding to $F$ and $\cu$. When $\cu \equiv \rr^d$ (the case considered here), it must be the case that $\|F(\vu^*)\|_{p^*} = 0.$ We will assume that there exists at least one \eqref{eq:SVI} solution, and will denote the set of all such solutions by $\cu^*.$ 

The \emph{Minty Variational Inequality} problem consists in finding $\vu^*$ such that:
\begin{equation}\label{eq:MVI}\tag{\textsc{mvi}}
    (\forall \vu \in \cu): \quad \innp{F(\vu), \vu^* - \vu} \leq 0,
\end{equation}
in which case $\vu^*$ is referred to as the \emph{weak solution} to the variational inequality corresponding to $F$ and $\cu$.   
If we assume that $F$ is monotone, then~\eqref{eq:mon-op} implies that every solution to~\eqref{eq:SVI} is also a solution to \eqref{eq:MVI}, and the two solution sets are equivalent. More generally, if $F$ is \emph{not} monotone, all that can be said is that the set of \eqref{eq:MVI} solutions is a subset of the set of \eqref{eq:SVI} solutions. In particular, \eqref{eq:MVI} solutions may not exist even when \eqref{eq:SVI} solutions exist. These facts follow from Minty's theorem (see, e.g.,~\citep[Chapter 3]{kinderlehrer2000introduction}).

We will \emph{not}, in general, be assuming that $F$ is monotone. Note that the Lipschitzness of $F$ on its own is not sufficient to guarantee that the problem is computationally tractable, as discussed in the introduction. Thus, additional structure is needed, which we introduce in the following. 

\paragraph{Weak {MVI} solutions.}

We define the class of problems with weak \eqref{eq:MVI} solutions as the class of problems in which $F$ satisfies the following assumption.

\begin{assumption}[Weak \textsc{mvi}]\label{assmpt:cohypo}
There exists $\vu^* \in \cu^*$ such that:
\begin{equation}\label{eq:assmpt-ncvx}\tag{\textsc{a}$_2$}
    (\forall \vu \in \rr^d):\quad 
    \innp{F(\vu), \vu - \vu^*} \geq -\frac{\rho}{2} \|F(\vu)\|_{p^*}^2,
\end{equation}
for some parameter $\rho \in \big[0, \frac{1}{4L}\big)$.   
\end{assumption}

We will only provide results for $\rho > 0$ in the case of the $\ell_2$ norm. For $p \neq 2,$ we will require a stronger assumption; namely, that an~\eqref{eq:MVI} solution exists, which holds when $\rho = 0$. 

\subsection{Example Settings Satisfying Assumption~\ref{assmpt:cohypo}}

The class of problems that have weak \eqref{eq:MVI} solutions in the sense of Assumption~\ref{assmpt:cohypo} generalizes other structured non-monotone variational inequality problems, as we discuss in this section.

When $\rho = 0,$ we recover the class of problems that have an~\eqref{eq:MVI} solution. This class  contains all unconstrained variationally coherent problems studied in, e.g.,~\citet{zhou2017stochastic,mertikopoulos2018optimistic}, which encompass all min-max problems with objectives that are bilinear, pseudo-convex-concave, quasi-convex-concave, and star-convex-concave. 

When $\rho > 0$ and $p=2$, Assumption~\ref{assmpt:cohypo} is implied by $F$ being $-\frac{\rho}{2}$-comonotone~\cite{bauschke2020generalized} or $\frac{\rho}{2}$-cohypomonotone~\cite{combettes2004proximal}, defined as follows:
\begin{align}
     &(\forall \vu, \vv \in \rr^d):\notag\\ &~\innp{F(\vu) - F(\vv), \vu - \vv} \geq - \frac{\rho}{2}\|F(\vu) - F(\vv)\|_2^2. \label{eq:cohypomonotone}
\end{align}
In particular, Assumption~\ref{assmpt:cohypo} is equivalent to  requiring that~\eqref{eq:cohypomonotone}  be satisfied for general $\vu$ and $\vv = \vu^*$, where $\vu^*$ is a solution to~\eqref{eq:SVI} (in which case $F(\vu^*) = \zeros$). Note that Assumption~\ref{assmpt:cohypo} does \emph{not} imply that a solution to~\eqref{eq:MVI} exists, unless $\rho=0$. It is further important to note that cohypomonotone operators arise as inverses of operators that only need to be Lipschitz-continuous. (In fact, even a weaker property suffices; see~\citet{bauschke2020generalized}.) This is particularly interesting as, combined with our main result, it implies that we can efficiently find zeros of inverses of Lipschitz-continuous operators, as long as those inverses are sufficiently Lipschitz, even though finding zeros of Lipschitz-continuous operators is computationally intractable, in general, as we have discussed.

It is interesting to note that Assumption~\ref{assmpt:cohypo} does not imply that, in the min-max setting, $f$ is convex-concave (or, more generally, that $F$ is monotone), even in a neighborhood of an~\eqref{eq:SVI} solution $\vu^* = \big[\subalign{\vx^* \\ \vy^*}\big]$, i.e., a stationary point of $f$. To see this, fix $\vy = \vy^*$ and consider $f(\vx, \vy^*)$ for $\vx$ in a small neighborhood of $\vx^*.$ Using the fact that a continuously-differentiable function is well-approximated by its linear approximation within small neighborhoods, all that we are able to deduce from Assumption~\ref{assmpt:cohypo} is that
\ifaistats
\else
\begin{align*}
    f(\vx^*, \vy^*) - f(\vx, \vy^*) \approx \innp{\Big[\subalign{\nabla_{\vx} f(\vx, \vy^*)\\ \nabla_{\vy} f(\vx, \vy^*)}\Big], \Big[\subalign{&\vx^* - \vx\\ &\vy^* - \vy^*}\Big]} \leq \frac{\rho}{2}\|\nabla f(\vx, \vy^*)\|_{p^*}^2. 
\end{align*}
\fi
In particular, Assumption~\ref{assmpt:cohypo} does not preclude that $f(\vx^*, \vy^*)$ is larger than $f(\vx, \vy^*)$; it only bounds how much larger it can be by a quantity proportional to $\|\nabla f(\vx, \vy^*)\|_{p^*}^2$. Compare this also to the Polyak-{\L}ojasiewicz condition (see, e.g., \citet{nouiehed2019solving,yang2020global}), which imposes the opposite inequality, namely, that $f(\vx, \vy^*) - f(\vx^*, \vy^*)$ is bounded {\em above} by a multiple of $\|\nabla f(\vx, \vy^*)\|_{p^*}^2$. 

One  way that a generic operator $F$ may satisfy Assumption~\ref{assmpt:cohypo} is when there is a constant $\lambda > 0$ such that for some $\vu^* \in \cu^*$ we have
\begin{align}
    (\forall \vu \in \rr^d) \innp{F(\vu), \vu - \vu^*} \geq - \frac{\lambda}{2}\|\vu - \vu^*\|_p^2, \label{eq:costis}
\end{align}
and when the operator $F$ does not plateau or become too close to a linear operator around $\vu^*;$ namely, $\|F(\vu) - F(\vu^*)\|_{p^*} \geq \mu \|\vu - \vu^*\|_{p}$. (Note that~\eqref{eq:costis} is always satisfied with $\lambda = 2L$ for $L$-Lipschitz operators, but we may need $\lambda$ to be smaller than $2L$).  Then Assumption~\ref{assmpt:cohypo} would be satisfied with $\rho = \frac{\lambda}{\mu}.$ For a min-max problem, assuming $f$ is twice differentiable, this would mean that the lowest eigenvalue of the symmetric part of the Jacobian of $\big[\subalign{\nabla_{\vx}f(\vx, \vy)\\ - \nabla_{\vy}f(\vx, \vy)}\big]$ is bounded below by $-\lambda/2$ in any direction $\vu - \vu^*$ and the function $f$ is sufficiently ``curved'' (not close to a linear or a constant function) around $\vu^* = \big[\subalign{\vx^* \\ \vy^*}\big].$

Finally, we discuss a concrete min-max application wherein there are no~\eqref{eq:MVI} solutions, but there do exist \eqref{eq:SVI} solutions satisfying the weak~\eqref{eq:MVI} condition of Assumption~\ref{assmpt:cohypo}. This application arises in the context of two-agent zero-sum reinforcement learning problems studied by many authors, including recently by~\citet{daskalakis2021independent}. In Section 5.1 of that work, the authors consider a special case of the general two-agent zero-sum RL problem, called von Neumann’s ratio game, for which they observe that, even on a random example, the \eqref{eq:MVI} solution set is empty, yet the extragradient method still converges in practice (albeit at a slower rate). Interestingly, it is easy to construct examples of the von Neumann ratio game for which no~\eqref{eq:MVI} solution exists, but the {weak}~\eqref{eq:MVI} condition of Assumption~\ref{assmpt:cohypo} does hold, and yet the stronger cohypomonotonicity condition of~\eqref{eq:cohypomonotone} does not hold. Indeed, one such example is obtained for the game shown in Proposition 2 of their paper, setting $s=1/2$ and $\epsilon=.49$. Here~\eqref{eq:MVI} fails, the weak~\eqref{eq:MVI} condition of Assumption~\ref{assmpt:cohypo} is satisfied, and cohypomonotonicity fails to hold, e.g., for $\vu=(\vx,\vy)=(0.1,0.3)$  and $\vv=(\vx’,\vy’)=(0.8,0.3)$.  To be clear, the von Neumann ratio game gives rise to a constrained min-max problem while our algorithm targets the unconstrained setting. While extending our result to the constrained setting remains open, our example here demonstrates that there is value in further studying the weak~\eqref{eq:MVI}  condition of Assumption~\ref{assmpt:cohypo} in the constrained setting as well. 
\subsection{Useful Definitions and Facts}

We now list some useful definitions and facts that will subsequently be used in our analysis. 
\ifaistats
\else
We start with a presentation of uniformly convex functions, convex conjugates, and Bregman divergence, and then specialize these basic facts to the $\ell_p$ setups used in Section~\ref{sec:extensions}.
\fi

\begin{definition}[Uniform convexity]\label{def:unif-cvxty}
Given $p \geq 2$, a differentiable function $\psi: \rr^d \to \rr \cup \{+ \infty\}$ is said to be $p$-uniformly convex w.r.t.~$\|\cdot\|$ and with constant $m$ if $\forall \vx, \vy \in \rr^d,$
$$
    \quad \psi(\vy) \geq \psi(\vx) + \innp{\nabla \psi(\vx), \vy - \vx} + \frac{m}{p}\|\vy - \vx\|^p.
$$
\end{definition}
Observe that when $p = 2,$ we recover the standard definition of strong convexity. Thus, uniform convexity is a generalization of strong convexity.

\ifaistats
\else
\begin{definition}[Convex conjugate]\label{def:cvx-conj}
Given a convex function $\psi: \rr^d \to \rr \cup \{+\infty\},$ its convex conjugate $\psi^*$ is defined by:
$$
    (\forall \vz \in \rr^d):\quad \psi^*(\vz) = \sup_{\vx \in \rr^d}\{\innp{\vz, \vx} - \psi(\vx)\}.
$$
\end{definition}

The following standard fact can be derived using the Fenchel-Young inequality, which states that $\forall \vx, \vz \in \rr^d: \psi(\vx) + \psi^*(\vz) \geq \innp{\vz, \vx},$ and it is a simple corollary of Danskin's theorem (see, e.g.,~\cite{bertsekas1971control,Bertsekas2003}).

\begin{fact}\label{fact:danskin}
Let $\psi: \rr^d \to \rr \cup \{+\infty\}$ be a closed convex proper function and let $\psi^*$ be its convex conjugate. Then,  $\forall \vg \in \partial \psi^*(\vz),$
$$
     \vg \in \argsup_{\vx \in \rr^d}\{\innp{\vz, \vx} - \psi(\vx)\},
$$
where $\partial \psi^*(\vz)$ is the subdifferential set (the set of all subgradients) of $\psi^*$ at point $\vz$. In particular, if $\psi^*$ is differentiable, then $\argsup_{\vx \in \rr^d}\{\innp{\vz, \vx} - \psi(\vx)\}$ is a singleton set and $\nabla \psi^*(\vz)$ is its only element.
\end{fact}
\fi

\begin{definition}[Bregman divergence]
Let $\psi:\rr^d \to \rr$ be a differentiable function. Then its Bregman divergence between points $\vx, \vy \in \rr^d$ is defined by
$$
    D_{\psi}(\vx, \vy) = \psi(\vx)- \psi(\vy) - \innp{\nabla \psi(\vy), \vx - \vy}.
$$
\end{definition}
\noindent
It is immediate that the Bregman divergence of a convex function is non-negative.

\paragraph{Useful facts for $\ell_p$ setups.} We now outline some useful auxiliary results used specifically in Section~\ref{sec:extensions}, where we study the case that $p$ is not necessarily equal to 2.

\begin{restatable}{proposition}{propsteptograd}\label{prop:step-to-grad}
Given, $\vz, \vu \in \rr^d$, $p \in (1, \infty)$ and $q \in \{p, 2\}$, let
$$
    \vw = \argmin_{\vv \in \rr^d} \Big\{\innp{\vz, \vv} + \frac{1}{q}\|\vu - \vv\|_p^q\Big\}.
$$
Then, for $p^* = \frac{p}{p-1},$ $q^* = \frac{q}{q-1}$:
$$
    \vw = \vu - \nabla \Big(\frac{1}{q^*}\|\vz\|_{p^*}^{q^*}\Big) \quad\text{ and } \quad \frac{1}{q}\|\vw - \vu\|_p^q = \frac{1}{q}\|\vz\|_{p^*}^{q^*}. 
$$
\end{restatable}
\ifaistats
\else
\begin{proof}
The statements in the proposition are simple corollaries of conjugacy of the functions $\psi(\vu) = \frac{1}{q}\|\vu\|_p^q$ and $\psi^*(\vz) = \frac{1}{q^{*}}\|\vz\|_{p^*}^{q^*}$. In particular, the first part follows from
$$
    \psi^*(\vz) = \sup_{\vv\in \rr^d}\{\innp{\vz, \vv} - \psi(\vv)\},
$$
by the definition of a convex conjugate and using that $ \frac{1}{q}\|\vu\|_p^q$ and $ \frac{1}{q^{*}}\|\vz\|_{p^*}^{q^*}$ are conjugates of each other, which are standard exercises in convex analysis for $q \in \{p, 2\}$ (see, e.g.,~\citep[Exercise 4.4.2]{borwein2004techniques} and \citep[Example 3.27]{boyd2004convex}). 

The second part follows by $\nabla \psi^*(\vz) = \arg\sup_{\vv\in \rr^d}\{\innp{\vz, \vv} - \psi(\vv)\},$ due to Fact~\ref{fact:danskin} ($\psi$ and $\psi^*$ are both continuously differentiable for $p \in (1, \infty)$). Lastly, $ \frac{1}{q}\|\vw - \vu\|_p^q = \frac{1}{q}\|\vz\|_{p^*}^{q^*}$ can be verified by setting $\vw = \vu - \nabla \big(\frac{1}{q^*}\|\vz\|_{p^*}^{q^*}\big).$
\end{proof}
\fi

Another useful result is the following proposition, which will allow us to relate Lipschitzness of $F$ to uniform convexity of the prox mapping $\frac{1}{q}\|\cdot\|_p^q$ in the definition of the algorithm. The ideas used in the proof can be found in the proofs of~\citep[Lemma 5.7]{d2018optimal},  \citep[Lemma 2]{nesterov2015universal}, and in \citep[Section 2.3]{devolder2014first}. \ifaistats\else The proof is provided for completeness.\fi
\begin{restatable}{proposition}{propsmoothub}\label{prop:smooth-ub}
For any $L > 0$, $\kappa > 0$, $q\geq \kappa,$ $t \geq 0$, and $\delta > 0$,
$$
    \frac{L}{\kappa}t^{\kappa} \leq \frac{\Lambda}{q}t^q + \frac{\delta}{2},
$$
where
\ifaistats
\else
$$
    \Lambda = \Big(\frac{2(q-\kappa)}{\delta q \kappa}\Big)^{\frac{q-\kappa}{\kappa}}L^{q/\kappa}.
$$
\fi
\end{restatable}
\ifaistats
\else
\begin{proof}
The proof is based on the Fenchel-Young inequality and the conjugacy of functions $\frac{|x|^r}{r}$ and $\frac{|y|^s}{s}$ for $r, s \geq 1$, $\frac{1}{r} + \frac{1}{s} = 1,$ which implies $xy \leq \frac{x^r}{r} + \frac{y^s}{s},$ $\forall x, y \geq 0.$ In particular, setting $r = q/\kappa,$ $s = q/(q-\kappa)$, and $x = t^{\kappa},$ we have
$$
    \frac{L}{\kappa}t^{\kappa} \leq \frac{L t^q}{q y} + \frac{L(q - \kappa)}{q\kappa} y^{\frac{\kappa}{q-\kappa}}.
$$
It remains to set $\frac{\delta}{2} = \frac{L(q - \kappa)}{q\kappa} y^{\frac{\kappa}{q-\kappa}},$ which, solving for $y,$ gives $y = \big(\frac{\delta q \kappa}{2 L (q - \kappa)}\big)^{q - \kappa}$, and verify that, under this choice, $\Lambda = \frac{L t^q}{q y}.$
\end{proof}
\fi

\section{Generalized Extragradient for Problems with Weak MVI Solutions}\label{sec:mod-eg}

In this section, we consider the setup with the Euclidean norm $\|\cdot\| = \|\cdot\|_2$, i.e., $p = 2.$ To address the class of problems with weak \eqref{eq:MVI} solutions (see Assumption~\ref{assmpt:cohypo}), we introduce the following generalization of the extragradient algorithm, to which we refer as  Extragradient+ (\egp).
\ifaistats
\else
\begin{equation}\label{eq:mod-eg}\tag{\egp}
\begin{gathered}
    \vub_k = \argmin_{\vu \in \rr^d}\Big\{\frac{a_k}{\beta} \innp{F(\vu_k), \vu - \vu_k} + \frac{1}{2}\|\vu - \vu_k\|^2\Big\} = \vu_k - \frac{a_k}{\beta}F(\vu_k),\\
    \vu_{k+1} = \argmin_{\vu \in \rr^d} \Big\{a_k \innp{F(\vub_k), \vu - \vu_k} + \frac{1}{2}\|\vu - \vu_k\|^2\Big\} = \vu_k - a_k F(\vub_k),
\end{gathered}    
\end{equation}
\fi
where $\beta \in (0, 1]$ is a parameter of the algorithm and $a_k > 0$ is the step size. When $\beta = 1,$ we recover standard \eg.

The analysis relies on the following merit (or gap) function:
\begin{equation}\label{eq:gap-def-nmon}
    h_k :=  a_k \Big(\innp{F(\vub_k), \vub_k - \vu^*} + \frac{\rho}{2}\|F(\vub_k)\|^2\Big),
\end{equation}
for some $\vu^* \in \cu^*$ for which $F$ satisfies Assumption~\ref{assmpt:cohypo}. Then Assumption~\ref{assmpt:cohypo} implies that $h_k \geq 0,$ $\forall k.$

The first (and main) step is to bound all $h_k$'s above, as in the following lemma.

\begin{restatable}{lemma}{egpBndhk}\label{lemma:eg+bound-on-h-k}
Let $F:\rr^d\to\rr^d$ be an arbitrary $L$-Lipschitz operator that satisfies Assumption~\ref{assmpt:cohypo} for some $\vu^* \in \cu^*$. Given an arbitrary initial point $\vu_0,$ let the sequences of points $\{\vu_i\}_{i\geq 1}$, $\{\vub_i\}_{i\geq 0}$ evolve according to \eqref{eq:mod-eg} for some $\beta \in (0, 1]$ and positive step sizes $\{a_i\}_{i\geq 0}.$ Then, for any $\gamma > 0$ and any $k \geq 0,$ we have:
\ifaistats
\else
\begin{equation}
    \begin{aligned}
        h_k \leq &\; \frac{1}{2}\|\vu^* - \vu_k\|^2 - \frac{1}{2}\|\vu^* - \vu_{k+1}\|^2 + \frac{a_k}{2}\big( \rho - a_k(1-\beta)\big)\|F(\vub_k)\|^2\\
    &+ \frac{{a_k}^2}{2\beta^2}\big(a_kL\gamma - \beta\big)\|F(\vu_k)\|^2
    + \frac{1}{2}\Big(\frac{a_k L}{\gamma} - \beta\Big)\|\vub_k - \vu_{k+1}\|^2,
    \end{aligned}
\end{equation}
\fi
where $h_k$ is defined as in Eq.~\eqref{eq:gap-def-nmon}.
\end{restatable}
\ifaistats
\else
\begin{proof}
Fix any $k \geq 0$ and write $h_k$ equivalently as
\begin{equation}\label{eq:mod-eg-change-in-gap}
\begin{aligned}
    h_k =& \;  a_k\innp{F(\vub_k), \vu_{k+1} - \vu^*} + a_k\innp{F(\vu_k), \vub_k - \vu_{k+1}} \\
    &+ a_k \innp{F(\vub_k) - F(\vu_k), \vub_k - \vu_{k+1}}
    + a_k \frac{\rho}{2}\|F(\vub_k)\|^2.
\end{aligned}
\end{equation}
The proof proceeds by bounding from above individual terms on the right-hand side of Eq.~\eqref{eq:mod-eg-change-in-gap}. For the first term, the first-order optimality in the definition of $\vu_{k+1}$ gives:
$$
    a_k F(\vub_k) + \vu_{k+1} - \vu_k = \zeros.
$$
Thus, we have
\begin{equation}\label{eq:mod-eg-ineq-1}
\begin{aligned}
    a_k\innp{F(\vub_k), \vu_{k+1} - \vu^*} &= -\innp{\vu_{k+1} - \vu_k, \vu_{k+1} - \vu^*}\\
    &= \frac{1}{2}\|\vu^* - \vu_k\|^2 - \frac{1}{2}\|\vu^* - \vu_{k+1}\|^2 - \frac{1}{2}\|\vu_k - \vu_{k+1}\|^2.
\end{aligned}
\end{equation}

For the second term on the right-hand side of Eq.~\eqref{eq:mod-eg-change-in-gap}, the first-order optimality in the definition of $\vub_k$ implies:
\begin{equation*}
    \frac{a_k}{\beta}\innp{F(\vu_k) + \vub_{k} - \vu_k, \vu_{k+1} - \vub_k} = 0,
\end{equation*}
which, similarly as for the first term, leads to:
\begin{equation}\label{eq:mod-eg-ineq-2}
    a_k\innp{F(\vu_k), \vub_k - \vu_{k+1}} = \frac{\beta}{2}\|\vu_k - \vu_{k+1}\|^2 - \frac{\beta}{2}\|\vu_k - \vub_k\|^2 - \frac{\beta}{2}\|\vu_{k+1} - \vub_k\|^2.
\end{equation}

For the third term on the right-hand side of Eq.~\eqref{eq:mod-eg-change-in-gap}, applying the Cauchy-Schwarz inequality, $L$-Lipschitzness of $F,$ and Young's inequality, respectively, we have:
\begin{align}
    a_k \innp{F(\vub_k) - F(\vu_k), \vub_k - \vu_{k+1}} & \leq a_k \|F(\vub_k) - F(\vu_k)\|\|\vub_k - \vu_{k+1}\|\notag\\
    &\leq a_k L \|\vub_k - \vu_k\|\|\vub_k - \vu_{k+1}\|\notag\\
    &\leq \frac{a_k L \gamma}{2}\|\vub_k - \vu_k\|^2 + \frac{a_k L}{2\gamma}\|\vub_k - \vu_{k+1}\|^2, \label{eq:mod-eg-ineq-3}
\end{align}
where the last inequality holds for any $\gamma > 0.$

Using the fact that $\vub_k - \vu_k = -\frac{a_k}{\beta}F(\vu_k)$, $\vu_{k+1} - \vu_k = -a_k F(\vub_k)$ and combining Eqs.~\eqref{eq:mod-eg-ineq-1}-\eqref{eq:mod-eg-ineq-3} with Eq.~\eqref{eq:mod-eg-change-in-gap}, we have:
\begin{align*}
    h_k \leq &\; \frac{1}{2}\|\vu^* - \vu_k\|^2 - \frac{1}{2}\|\vu^* - \vu_{k+1}\|^2 
    + \frac{a_k}{2}\big( \rho - a_k(1-\beta)\big)\|F(\vub_k)\|^2\\
    &+ \frac{{a_k}^2}{2\beta^2}\big(a_kL\gamma - \beta\big)\|F(\vu_k)\|^2
    + \frac{1}{2}\Big(\frac{a_k L}{\gamma} - \beta\Big)\|\vub_k - \vu_{k+1}\|^2,
\end{align*}
as claimed.
\end{proof}
\fi

Using Lemma~\ref{lemma:eg+bound-on-h-k}, we can now draw conclusions about the convergence of \egp~by choosing parameters $\beta, \gamma$ and the step sizes $a_k$ to guarantee that $h_k < \frac{1}{2}\|\vu^* - \vu_k\|^2 - \frac{1}{2}\|\vu^* - \vu_{k+1}\|^2$ as long as $\|F(\vub_k)\|\neq 0$. 

\begin{theorem}\label{thm:convergence-of-eg+}
Let $F:\rr^d\to\rr^d$ be an arbitrary $L$-Lipschitz operator that satisfies Assumption~\ref{assmpt:cohypo} for some $\vu^* \in \cu^*$. Given an arbitrary initial point $\vu_0 \in \rr^d,$ let the sequences of points $\{\vu_i\}_{i\geq 1}$, $\{\vub_i\}_{i\geq 0}$ evolve according to \eqref{eq:mod-eg} for $\beta = \frac{1}{2}$ and $a_k = \frac{1}{2L}.$ Then:
\begin{itemize}
    \item[(i)]  all accumulation points of $\{\vub_k\}_{k\geq 0}$ are  in $\cu^*$.
    \item[(ii)] for all $k\geq 1:$ 
    $$
        \frac{1}{k+1}\sum_{i=0}^k \|F(\vub_i)\|^2 \leq \frac{2L \|\vu_0 - \vu^*\|^2}{(k+1)(1/(4L) - \rho)}.
    $$
    In particular, we have that 
    $$
        \min_{0\leq i\leq k} \|F(\vub_i)\|^2 \leq \frac{2L \|\vu_0 - \vu^*\|^2}{(k+1)(1/(4L) - \rho)}
    $$
    and 
    $$
        \ee_{i \sim \mathrm{Unif}\{0, \dots, k\}}\big[\|F(\vub_i)\|^2\big] \leq \frac{2L \|\vu_0 - \vu^*\|^2}{(k+1)(  1/(4L) - \rho)},
    $$
    where $i \sim \mathrm{Unif}\{0, \dots, k\}$ denotes an index $i$ chosen uniformly at random from the set $\{0, \dots, k\}.$
\end{itemize}
\end{theorem}
\begin{proof}
Applying Lemma~\ref{lemma:eg+bound-on-h-k} with the choice of $a_k$ and $\beta$ from the theorem statement and with $\gamma = 1,$ we get
\ifaistats
\else
\begin{align*}
    h_k \leq \frac{1}{2}\|\vu^* - \vu_k\|^2 - \frac{1}{2}\|\vu^* - \vu_{k+1}\|^2
    + \frac{1}{4L}\Big(\rho - \frac{1}{4L}\Big)\|F(\vub_k)\|^2.
\end{align*}
\fi
By Assumption~\ref{assmpt:cohypo}, $\rho < \frac{1}{4L},$ and, thus, the constant multiplying $\|F(\vub_k)\|^2$ is strictly negative.

As $h_k \geq 0$ (by Assumption~\ref{assmpt:cohypo}), we can conclude that
\ifaistats
\else
\begin{equation}\label{eq:eg+conv-main}
    \frac{1}{2}\|\vu^* - \vu_{k+1}\|^2 - \frac{1}{2}\|\vu^* - \vu_k\|^2 \leq - \frac{1}{4L}\Big(\frac{1}{4L} - \rho\Big)\|F(\vub_k)\|^2 \leq 0.
\end{equation}
\fi
As $\frac{1}{4L}\Big(\frac{1}{4L} - \rho\Big) > 0,$ Eq.~\eqref{eq:eg+conv-main} implies that $\|F(\vub_k)\|$ converges to zero as $k \to \infty.$ Further, as $\vub_k - \vu_k = -\frac{a_k}{\beta}F(\vu_k),$ using triangle inequality and $F(\vu^*) = \zeros:$
\begin{equation}\label{eq:bound-on diam}
\begin{aligned}
    \|\vub_k - \vu^*\| &\leq \|\vu_k - \vu^*\| + \frac{a_k}{\beta}\|F(\vu_k) - F(\vu^*)\|\\
    &\leq \Big(1 + L\frac{a_k}{\beta}\Big)\|\vu_k - \vu^*\| = 2 \|\vu_k - \vu^*\|,
\end{aligned}
\end{equation}
where we have used that $F$ is $L$-Lipschitz. Now, as $\|\vu_k - \vu^*\|$ is bounded (by $\|\vu_0 - \vu^*\|,$ from Eq.~\eqref{eq:eg+conv-main}), it follows that the sequence $\{\vub_k\}$ is bounded as well, and thus has a converging subsequence. Let $\{\vub_{k_i}\}$ be any converging subsequence of $\{\vub_k\}$ and let $\vub^*$ be its corresponding accumulation point. Then, as $\|F(\vub_k)\|$ converges to zero as $k \to \infty,$ it follows that $\|F(\vub_{k_i})\|$ converges to zero as $i \to \infty,$ and so it must be $\vub^* \in \cu^*.$ 


For Part (ii), telescoping Eq.~\eqref{eq:eg+conv-main}, we get:
\ifaistats
\else
\begin{align*}
    \frac{1}{4L}\Big(\frac{1}{4L} - \rho\Big)\sum_{i=0}^k \|F(\vub_i)\|^2 &\leq \frac{1}{2}\|\vu_0 - \vu^*\|^2 - \frac{1}{2}\|\vu_{k+1} - \vu^*\|^2 \\
    &\leq \frac{1}{2}\|\vu_0 - \vu^*\|^2.
\end{align*}
\fi
Rearranging the last inequality:
$$
    \frac{1}{k+1}\sum_{i=0}^k \|F(\vub_i)\|^2 \leq \frac{2L \|\vu_0 - \vu^*\|^2}{(k+1)(1/(4L) - \rho)}.
$$
It remains to observe that  $$\mathbb{E}_{i\sim\mathrm{Unif}\{0,...,k\}}[\|F(\vub_i)\|^2] = \frac{1}{k+1}\sum_{i=0}^k \|F(\vub_i)\|^2$$
and $\frac{1}{k+1}\sum_{i=0}^k \|F(\vub_i)\|^2 \geq \min_{0\leq i \leq k}\|F(\vub_i)\|^2$.
\end{proof}

\begin{remark}
Due to Eq.~\eqref{eq:bound-on diam}, we have that all the iterates of \egp~with the parameter setting as in Theorem~\ref{thm:convergence-of-eg+}
remain in the ball centered at $\vu^*$ and of radius at most $2\|\vu_0 - \vu^*\|$. Thus, Assumption~\ref{assmpt:cohypo} does not need to hold globally for the result to apply; it suffices that it only applies locally to points from the ball around $\vu^*$ with radius $2\|\vu_0 - \vu^*\|$.
\end{remark}

\begin{remark}
It is possible to obtain similar convergence results as those of Theorem~\ref{thm:convergence-of-eg+} under different parameter choices. In particular, for $\gamma \in (0, 1],$ it suffices that $a_k \leq \frac{\beta \gamma}{L}$ and $\rho < a_k(1-\beta).$ We settled on the choice made in Theorem~\ref{thm:convergence-of-eg+} as it is simple and requires tuning only one parameter, $L$.
\end{remark}

\begin{remark}\label{rem:SVI-existence}
Note that, in fact, we did not need to assume that $\vu^*$ from Assumption~\ref{assmpt:cohypo} is from $\cu^*;$ it could have been any point from $\rr^d$ for which Assumption~\ref{assmpt:cohypo} is satisfied. All that would change in the proof of Theorem~\ref{thm:convergence-of-eg+} is that in Eq.~\eqref{eq:bound-on diam}, using $\|F(\vu_k)\| \leq \|F(\vu_k) - F(\vu^*)\| + \|F(\vu^*)\|$ (by triangle inequality) we would have $2\|\vu_k - \vu^*\| + \frac{1}{L}\|F(\vu^*)\|$ on the right-hand side. Since $\vu^* \in \rr^d$ and $F$ is Lipschitz-continuous, if $F$ is bounded at any point $\vu \in \rr^d$, $\|F(\vu^*)\|$ is bounded as well. Thus, we can still conclude that $\|\vub_k - \vu^*\|$ is bounded and proceed with the rest of the proof. An interesting consequence of this observation and the proof of Theorem~\ref{thm:convergence-of-eg+} is that  Assumption~\ref{assmpt:cohypo} \emph{guarantees} existence of an \eqref{eq:SVI} solution. 
\end{remark}

\section{Extensions: $\ell_p$ Norms and Stochastic Setups}\label{sec:extensions}


In this section, we show how to extend the results of Section~\ref{sec:mod-eg} to non-Euclidean, $\ell_p$-normed setups (for $\rho = 0$)
and stochastic evaluations of $F$. In particular, we let $\|\cdot\| = \|\cdot\|_p$ for $p \in (1, \infty)$\footnote{Note that the norms $\|\cdot\|_1$ and $\|\cdot\|_{\infty}$ are within a constant factor of the $\ell_p$-norm for $p = 1 + \frac{1}{\log(d)}$ and $p = \log(d),$ respectively, and so taking $p \in (1, \infty)$ is w.l.o.g.---for any $p < 1 + \frac{1}{\log(d)}$ or $p > \log(d)$, we can run the algorithm with $p = 1 + \frac{1}{\log d}$ or $p = {\log d}$, losing at most a constant factor in the convergence bound.} and $p^* = \frac{p}{p-1}.$ Further, we let $\tF$ denote the stochastic estimate of $F$ that at iteration $k$ satisfies:
\ifaistats
\else
\begin{equation}\label{eq:stoch-est-F}
    \begin{gathered}
        \ee[\tF(\vub_{k})| \cfb_k] = F(\vub_{k}), \quad  \ee[\|\tF(\vub_{k}) - F(\vub_{k})\|_*^2 | \cfb_k] \leq {\bar{\sigma}_k}^2\\
        \ee[\tF(\vu_{k+1})| \cf_{k+1}] = F(\vu_{k+1}), \quad \ee[\|\tF(\vu_{k+1}) - F(\vu_{k+1})\|_*^2 |\cf_{k+1}] \leq {\sigma_{k+1}}^2,
    \end{gathered}
\end{equation}
\fi
where $\cf_k$ and $\cfb_k$ denote the natural filtrations, including all the randomness up to the construction of points $\vu_k$ and $\vub_k,$ respectively, and ${\bar{\sigma}_k}^2, {\sigma_{k+1}}^2$ are the variance constants. Observe that $\cf_k \subseteq \cfb_k$ and $\cfb_k \subseteq \cf_{k+1}.$ To simplify the notation, we denote:
\begin{equation}\label{eq:noise-vectors}
    \vetab_k = \tF(\vub_{k}) - F(\vub_{k}), \; \veta_{k+1} = \tF(\vu_{k+1}) - F(\vu_{k+1}).
\end{equation}

The variant of the method we consider here is stated as follows:
\begin{equation}\label{eq:egp+}\tag{\egpp}
    \begin{gathered}
        \vub_k = \argmin_{\vu \in \rr^d}\Big\{\frac{a_k}{\beta} \innp{\tF(\vu_k), \vu - \vu_k} +  \frac{1}{q}\|\vu - \vu_k\|_p^q\Big\},\\
        \vu_{k+1} = \argmin_{\vu \in \rr^d}\Big\{a_k \innp{ \tF(\vub_k), \vu - \vu_k} + \phi_p(\vu, \vu_k)\Big\},
    \end{gathered}
\end{equation}
where
\begin{equation}\label{eq:q-def}
    q = \begin{cases}
            2, &\text{ if } p \in (1, 2], \\
            p^* = \frac{p}{p-1}, &\text{ if } p \in (2, \infty)
    \end{cases}
\end{equation}
and 
\begin{equation}\label{eq:phi-p-def}
    \phi_p(\vu, \vu_k) = \begin{cases}
                            D_{\frac{1}{2}\|\cdot-\vu_0\|_p^2}(\vu, \vu_k), \text{ if } p \in (1, 2],\\
                            \frac{1}{p}\|\vu - \vu_k\|_p^p, \text{ if } p \in (2, \infty).
    \end{cases}
\end{equation}

Notice that for $p = 2$, \egpp~is equivalent to \egp. Thus, \egpp~generalizes \egp~to arbitrary $\ell_p$ norms. However, \egpp~is different from the standard Extragradient or Mirror-Prox, for two reasons. First is that, as is the case for \egp, the step sizes that determine $\vub_k$ and $\vu_{k+1}$ (i.e., $a_k/\beta$ and $a_k$) are not the same in general, as we could (and will) choose $\beta \neq 1.$ Second, unless $p = q = 2,$ the function $\frac{1}{q}\|\vu - \vu_k\|_p^q$ in the definition of the algorithm is \emph{not} a Bregman divergence between points $\vu$ and $\vu_k$ of any function $\psi.$ Further, when $p > 2,$ $\frac{1}{q}\|\vu - \vu_k\|_p^q$ is \emph{not} strongly convex. Instead, it is \emph{$p$-uniformly convex} with constant 1. Additionally, no function whose gap between the maximum and the minimum value is bounded by a constant on any ball of constant radius can have constant of strong convexity w.r.t.~$\|\cdot\|_p$ that is larger than $O(\frac{1}{d^{1-2/p}})$~\cite{d2018optimal}. When $p \in (1, 2],$ $\frac{1}{q}\|\vu - \vu_k\|_p^q$ is strongly convex with constant $p-1$~\cite{nemirovski2004regular}. We let $m_p$ denote the constant of strong/uniform convexity of $\frac{1}{q}\|\vu - \vu_k\|_p^q,$ that is:
\begin{equation}\label{eq:mp}
    m_p = \max\{p-1, 1\}.
\end{equation}
Observe that
\begin{equation}\label{eq:phi-p-unif-cvx}
    \phi_p(\vu, \vu_k) \geq \frac{m_p}{q}\|\vu - \vu_k\|_p^q.
\end{equation}
This is immediate for $p > 2,$ by the definition of $\phi_p$ and using that $q = p$ and $m_p = 1$ when $p > 2.$ For $p \in (1, 2],$ we have that $q = 2,$ and Eq.~\eqref{eq:phi-p-unif-cvx} follows by strong convexity of $\frac{1}{2}\|\cdot\|_p^2.$

As in the case of Euclidean norms, the analysis relies on the following merit function:
\begin{equation}\label{eq:h-k-egpp}
    h_k := a_k \Big(\innp{F(\vub_k), \vub_k - \vu^*} + \frac{\rho}{2}\|F(\vub_k)\|_{p^*}^2\Big).
\end{equation}
Moreover, as before, Assumption~\ref{assmpt:cohypo} guarantees that $h_k \geq 0,$ $\forall k.$

\ifaistats
\else
We start by first proving a lemma that holds for generic choices of algorithm parameters $a_k$ and $\beta.$ We will then use this lemma to deduce the convergence bounds for different choices of $p > 1$ and both deterministic and the stochastic oracle access to $F.$

\begin{lemma}\label{lemma:bound-on-hk-egpp}
Let $p > 1$ and let $F:\rr^d\to\rr^d$ be an arbitrary $L$-Lipschitz operator w.r.t.~$\|\cdot\|_p$ that satisfies Assumption~\ref{assmpt:cohypo} for some $\vu^* \in \cu^*$. Given an arbitrary initial point $\vu_0,$ let the sequences of points $\{\vu_i\}_{i\geq 1}$, $\{\vub_i\}_{i\geq 0}$ evolve according to \eqref{eq:egp+} for some $\beta \in (0, 1]$ and positive step sizes $\{a_i\}_{i\geq 0}.$ Then, for any $\gamma > 0$ and any $k \geq 0$:
\begin{equation}\notag
    \begin{aligned}
        h_k \leq&\; -a_k \innp{\vetab_k, \vub_k - \vu^*} - a_k\innp{\vetab_k - \veta_k, \vub_k - \vu_{k+1}} + \frac{a_k \rho}{2}\|F(\vub_k)\|_{p^*}^2 \\
     &+ \phi_p(\vu^*, \vu_k) - \phi_p(\vu^*, \vu_{k+1}) +  \frac{\beta - m_p}{q}\|\vu_{k+1} - \vu_k\|_p^q\\
     &+ \frac{a_k \Lambda_k \gamma - \beta}{q}\|\vub_k - \vu_k\|_{p}^q + \frac{a_k \Lambda_k/\gamma - \beta m_p}{q}\|\vub_k - \vu_{k+1}\|_p^q + a_k \delta_k,
    \end{aligned}
\end{equation}
where $h_k$ is defined as in Eq.~\eqref{eq:h-k-egpp}, $\delta_k$ is any positive number, and $\Lambda_k = \Big(\frac{q-2}{\delta_k q }\Big)^{\frac{q-2}{2}}L^{q/2}$. When $q = 2,$ the statement also  holds with $\delta_k = 0$ and $\Lambda_k = L.$
\end{lemma}
\begin{proof}
We begin the proof by writing $h_k$ equivalently as:
\begin{equation}\label{eq:hk-equiv-egpp}
    \begin{aligned}
        h_k =&\; a_k \innp{\tF(\vub_k), \vub_k - \vu^*} -a_k \innp{\vetab_k, \vub_k - \vu^*} + \frac{a_k \rho}{2}\|F(\vub_k)\|_{p^*}^2\\
        =&\; a_k \innp{\tF(\vub_k), \vu_{k+1} - \vu^*} + a_k \innp{\tF(\vu_k), \vub_k - \vu_{k+1}} \\
        &+ a_k \innp{\tF(\vub_k) - \tF(\vu_k), \vub_k - \vu_{k+1}} -a_k \innp{\vetab_k, \vub_k - \vu^*} + \frac{a_k \rho}{2}\|F(\vub_k)\|_{p^*}^2.
    \end{aligned}
\end{equation}
The proof now proceeds by bounding individual terms on the right-hand side of the last equality.

Let $M_{k+1}(\vu) = a_k \innp{\nabla \tF(\vub_k), \vu - \vu_k} + \phi_p(\vu, \vu_k),$ so that $\vu_{k+1} = \argmin_{\vu \in \rr^d}M_{k+1}(\vu)$. By the definition of Bregman divergence of $M_{k+1}:$
$$
    M_{k+1}(\vu^*) = M_{k+1}(\vu_{k+1}) + \innp{\nabla M_{k+1}(\vu_{k+1}), \vu^* - \vu_{k+1}} + D_{M_{k+1}}(\vu^*, \vu_{k+1}).
$$
As $\vu_{k+1} = \argmin_{\vu \in \rr^d}M_{k+1}(\vu)$, we have $\nabla M_{k+1}(\vu_{k+1}) = \zeros.$ Further, $D_{M_{k+1}}(\vu^*, \vu_{k+1}) = D_{\phi_p(\cdot, \vu_k)}(\vu^*, \vu_{k+1}).$ When $p \leq 2,$ $\phi_p$ itself is a Bregman divergence, and we have  $D_{M_{k+1}}(\vu^*, \vu_{k+1}) = \phi_p(\vu^*, \vu_{k+1}).$ When $p > 2,$ $\phi_p(\vu, \vu_k) = \frac{1}{p}\|\vu - \vu_k\|_p^p$, and as $\phi_p$ is $p$-uniformly convex with constant 1, it follows that $D_{M_{k+1}}(\vu^*, \vu_{k+1}) \geq \frac{1}{p}\|\vu^* - \vu_{k+1}\|_p^p = \phi_p(\vu^*, \vu_{k+1}).$ Thus:
$$
    M_{k+1}(\vu^*) \geq M_{k+1}(\vu_{k+1}) + \phi_p(\vu^*, \vu_{k+1}). 
$$
Equivalently, applying the definition of $M_{k+1}(\cdot)$ to the last inequality:
\begin{equation}\label{eq:egpp-1}
\begin{aligned}
    a_k \innp{\nabla \tF(\vub_k), \vu_{k+1} - \vu^*} &\leq \phi_p(\vu^*, \vu_k) - \phi_p(\vu^*, \vu_{k+1}) - \phi_p(\vu_{k+1, \vu_k})\\
    &\leq \phi_p(\vu^*, \vu_k) - \phi_p(\vu^*, \vu_{k+1}) - \frac{m_p}{q}\|\vu_{k+1} - \vu_k\|_p^q,
\end{aligned}
\end{equation}
where the last inequality follows from Eq.~\eqref{eq:phi-p-unif-cvx}.

Now let $\bar{M}_k(\vu) = \frac{a_k}{\beta} \innp{\tF(\vu_k), \vu - \vu_k} +  \frac{1}{q}\|\vu - \vu_k\|_p^q$ so that $\vub_k = \argmin_{\vu \in \rr^d}\bar{M}_k(\vu).$ By similar arguments as above,
\begin{align*}
    \bar{M}_k(\vu_{k+1}) &= \bar{M}_k(\vub_k) + \innp{\nabla \bar{M}_k(\vub_k), \vu_{k+1} - \vub_k} + D_{M_k}(\vu_{k+1}, \vub_k)\\
    &\geq \bar{M}_k(\vub_k) + \frac{m_p}{q}\|\vu_{k+1} - \vub_k\|_p^q,
\end{align*}
where the inequality is by $\nabla \bar{M}_k(\vub_k) = \zeros$ and the fact that $\frac{1}{q}\|\cdot\|_p^q$ is $q$-uniformly convex w.r.t.~$\|\cdot\|_p$ with constant $m_p$, by the choice of $q$ from Eq.~\eqref{eq:q-def}. Applying the definition of $\bar{M}_k(\vu)$ to the last inequality:
\begin{equation}\label{eq:egpp-2}
    a_k\innp{\tF(\vu_k), \vub_k - \vu_{k+1}} \leq \frac{\beta}{q}\big( \|\vu_{k+1} - \vu_k\|_p^q - \|\vub_k - \vu_k\|_p^q - m_p\|\vu_{k+1} - \vub_k\|_p^q\big).
\end{equation}

The remaining term that we need to bound is $\innp{\tF(\vub_k) - \tF(\vu_k), \vub_k - \vu_{k+1}}.$ Using the definitions of $\vetab_k, \veta_k,$ we have:
\begin{align*}
    \innp{\tF(\vub_k) - \tF(\vu_k), \vub_k - \vu_{k+1}} &= \innp{F(\vub_k) - F(\vu_k), \vub_k - \vu_{k+1}} - \innp{\vetab_k - \veta_k, \vub_k - \vu_{k+1}}\\
    &\stackrel{(i)}{\leq} - \innp{\vetab_k - \veta_k, \vub_k - \vu_{k+1}} + \|F(\vub_k) - F(\vu_k)\|_{p^*}\|\vub_k - \vu_{k+1}\|_p\\
    &\stackrel{(ii)}{\leq} - \innp{\vetab_k - \veta_k, \vub_k - \vu_{k+1}} + L\|\vub_k - \vu_k\|_{p}\|\vub_k - \vu_{k+1}\|_p\\
    &\stackrel{(iii)}{\leq} - \innp{\vetab_k - \veta_k, \vub_k - \vu_{k+1}} + \frac{L \gamma}{2}\|\vub_k - \vu_k\|_{p}^2 + \frac{L}{2\gamma}\|\vub_k - \vu_{k+1}\|_p^2,
\end{align*}
where $(i)$ is by H\"{o}lder's inequality, $(ii)$ is by $L$-Lipschitzness of $F,$ and $(iii)$ is by Young's inequality, which holds for any $\gamma > 0.$ Now, let $\delta_k > 0$ and $\Lambda_k = \Big(\frac{2(q-\kappa)}{\delta_k q \kappa}\Big)^{\frac{q-\kappa}{\kappa}}L^{q/\kappa}.$ Then, applying Proposition~\ref{prop:smooth-ub} to the last two terms in the last inequality:
\begin{equation}\label{eq:egpp-3}
    \begin{aligned}
        \innp{\tF(\vub_k) - \tF(\vu_k), \vub_k - \vu_{k+1}} \leq &\; - \innp{\vetab_k - \veta_k, \vub_k - \vu_{k+1}}\\
        &+ \frac{\Lambda_k \gamma}{q}\|\vub_k - \vu_k\|_{p}^q + \frac{\Lambda_k}{q\gamma}\|\vub_k - \vu_{k+1}\|_p^q + \delta_k.
    \end{aligned}
\end{equation}
Observe that when $q = 2,$ there is no need to apply Proposition~\ref{prop:smooth-ub}, and the last inequality is satisfied with $\delta_k = 0$ and $\Lambda_k = L.$

Combining Eqs.~\eqref{eq:egpp-1}-\eqref{eq:egpp-3} with Eq.~\eqref{eq:hk-equiv-egpp}, we have:
\begin{align*}
     h_k \leq&\; -a_k \innp{\vetab_k, \vub_k - \vu^*} - a_k\innp{\vetab_k - \veta_k, \vub_k - \vu_{k+1}} + \frac{a_k \rho}{2}\|F(\vub_k)\|_{p^*}^2 \\
     &+ \phi_p(\vu^*, \vu_k) - \phi_p(\vu^*, \vu_{k+1}) +  \frac{\beta - m_p}{q}\|\vu_{k+1} - \vu_k\|_p^q\\
     &+ \frac{a_k \Lambda_k \gamma - \beta}{q}\|\vub_k - \vu_k\|_{p}^q + \frac{a_k \Lambda_k/\gamma - \beta m_p}{q}\|\vub_k - \vu_{k+1}\|_p^q + a_k \delta_k,
\end{align*}
as claimed.
\end{proof}

We are now ready to state and prove the main convergence bounds. For simplicity, we start with the case of exact oracle access to $F$. We then show that we can build on this result by separately bounding the error terms due to the variance of the stochastic estimates $\tF.$

\fi 

\paragraph{Deterministic oracle access.} The main result is summarized in the following theorem.

\begin{restatable}{theorem}{thmegppdet}\label{thm:egpp_deterministic}
Let $p > 1$ and let $F:\rr^d\to\rr^d$ be an arbitrary $L$-Lipschitz operator w.r.t.~$\|\cdot\|_p$ that satisfies Assumption~\ref{assmpt:cohypo} {with $\rho = 0$} for some $\vu^* \in \cu^*$. Assume that we are given oracle access to the exact evaluations of $F,$ i.e., $\vetab_i = \veta_i = \zeros,$ $\forall i.$ Given an arbitrary initial point $\vu_0 \in \rr^d,$ let the sequences of points $\{\vu_i\}_{i\geq 1}$, $\{\vub_i\}_{i\geq 0}$ evolve according to \eqref{eq:egp+} for  $\beta \in (0, 1]$ and step sizes $\{a_i\}_{i\geq 0}$ specified below. Then, we have:
\begin{itemize}
    \item[(i)] Let $p \in (1, 2]$. If $\beta = m_p = p-1,$ $a_k = \frac{{m_p}^{3/2}}{2 L}$, then 
    all accumulation points of $\{\vu_k\}_{k\geq 0}$ are in $\cu^*,$
    and, furthermore $\forall k \geq 0$:
    \begin{align*}
        &\frac{1}{k+1}\sum_{i=0}^k \|F(\vu_i)\|_{p^*}^2 \leq \frac{16 L^2 \phi_p(\vu^*, \vu_0)}{{m_p}^2 (k+1)}\\
        &\hspace{1.5cm}= O\Big(\frac{L^2\|\vu^* - \vu_0\|_p^2}{{(p-1)}^2(k+1)}\Big).
    \end{align*}
    In particular, within $k = O\big(\frac{L^2 \|\vu^* - \vu_0\|_p^2}{{(p-1)}^2\epsilon^2}\big)$  iterations \egpp~can output a point $\vu$ with $\|F(\vu)\|_{p^*}\leq \epsilon.$
    \item[(ii)] Let $p \in (2, \infty)$. If $\beta = \frac{1}{2},$ $\delta_k = \delta > 0$, $\Lambda = \big(\frac{q-2}{\delta q }\big)^{\frac{q-2}{2}}L^{q/2}$, and $a_k = \frac{1}{2\Lambda} = a$, then, 
    $\forall k \geq 0$:
    $$
        \frac{1}{k+1}\sum_{i=0}^k \|F(\vub_i)\|_{p^*}^{p^*} \leq \frac{2\|\vu^* - \vu_0\|_p^p}{a^{p^*}(k+1)} + \frac{2p \delta}{a^{p^* - 1}}.
    $$
    In particular, for any $\epsilon > 0,$ there is a choice of $\delta = \frac{\epsilon^2}{C_pL},$ where $C_p$ is a constant that only depends on $p$, such that \egpp~can output a point $\vu$ with $\|F(\vu)\|_{p^*} \leq \epsilon$ in at most
    $$
        k = O_p\bigg(\Big(\frac{L\|\vu^* - \vu_0\|_p}{\epsilon}\Big)^p\bigg)
    $$
    iterations. Here, the $O_p$ notation hides constants that only depend on $p.$
\end{itemize}
\end{restatable}
\ifaistats
\else
\begin{proof}
Observe that, as $\vetab_i = \veta_i = \zeros,$ $\forall i \geq 0$ and $\rho = 0,$ Lemma~\ref{lemma:bound-on-hk-egpp} and the definition of $h_k$ give:
\begin{equation}\label{eq:egpp-det-h_k}
    \begin{aligned}
        0\leq h_k \leq&\; \phi_p(\vu^*, \vu_k) - \phi_p(\vu^*, \vu_{k+1}) +  \frac{\beta - m_p}{q}\|\vu_{k+1} - \vu_k\|_p^q \\
     &+ \frac{a_k \Lambda_k \gamma - \beta}{q}\|\vub_k - \vu_k\|_{p}^q + \frac{a_k \Lambda_k/\gamma - \beta m_p}{q}\|\vub_k - \vu_{k+1}\|_p^q + a_k \delta_k.
    \end{aligned}
\end{equation}

\noindent\textbf{Proof of Part (i).} In this case, we can set $\delta_k = 0$ (see Lemma~\ref{lemma:bound-on-hk-egpp}), $\Lambda_k = L,$  and $q = 2.$ Therefore, setting $\beta = m_p,$ $a_k = \frac{{m_p}^{3/2}}{2 L}$, and $\gamma = \frac{1}{\sqrt{m_p}}$ we get from Eq.~\eqref{eq:egpp-det-h_k} that
\begin{equation}
    \begin{aligned}\label{eq:h-k-egpp-rho-0}
        \phi_p(\vu^*, \vu_{k+1}) \leq \phi_p(\vu^*, \vu_k) - \frac{m_p}{4}\|\vub_k - \vu_k\|_p^2.
    \end{aligned}
\end{equation}
It follows that $\|\vub_k - \vu_k\|_p^2$ converges to zero as $k \to \infty.$ By the definition of $\vub_k$ and  Proposition~\ref{prop:step-to-grad}, $\frac{1}{2}\|\vub_k - \vu_k\|_p^2 = \frac{{a_k}^2}{2\beta^2}\|F(\vu_k)\|_{p^*}^2,$ and so $\|F(\vu_k)\|_{p^*}$ converges to zero as $k \to  \infty.$ Further, as $\phi_p(\vu^*, \vu_k) \leq \phi_p(\vu^*, \vu_0)< \infty$ and $\phi_p(\vu^*, \vu_k)\geq \frac{m_p}{2}\|\vu^* - \vu_k\|_p^2,$ $m_p>0,$ it follows that $\|\vu^* - \vu_k\|_p$ is bounded, and, thus, $\{\vu_k\}_{k\geq 0}$ is a bounded sequence. The proof that all accumulation points of $\{\vu_k\}_{k\geq 0}$ are in $\cu^*$ 
is standard and omitted (see the proof of Theorem~\ref{thm:convergence-of-eg+} for a similar argument). 

To bound $\frac{1}{k+1}\sum_{i=0}^k \|F(\vu_i)\|_{p^*}^2$, we  telescope the inequality from Eq.~\eqref{eq:h-k-egpp-rho-0} to get:
\begin{align*}
m_p\sum_{i=0}^k \|\vub_i - \vu_i\|_p^2 \leq 4(\phi_p(\vu^*, \vu_0) - \phi_p(\vu^*, \vu_{k+1}))\leq 4 \phi_p(\vu^*, \vu_0).    
\end{align*}
To complete the proof of this part, it remains to use the fact that $\|\vub_i - \vu_i\|_p^2 = \frac{{a_k}^2}{\beta^2}\|F(\vu_i)\|_{p^*}^2$ (already argued above), the definitions of $a_k$ and $\beta,$ and $m_p = p-1.$ The bound on $\phi_p(\vu^*, \vu_0)$ follows from the definition of $\phi_p$ in this case. In particular, if we denote $\psi(\vu) = \frac{1}{2}\|\vu - \vu_0\|_p^2,$ then $\phi_p(\vu^*, \vu_0) = D_{\psi}(\vu^*, \vu_0).$ Using the definition of Bregman divergence and the fact that, for this choice of $\psi,$ we have $\|\nabla \psi(\vu)\|_{p^*} = \|\vu - \vu_0\|_p,$ $\forall \vu \in \rr^d,$ (see the last part of  Proposition~\ref{prop:step-to-grad}) it follows that:
\begin{align*}
    \phi_p(\vu^*, \vu_0) &= \frac{1}{2}\|\vu^* - \vu_0\|_p^2 - \frac{1}{2}\|\vu_0 - \vu_0\|_p^2 - \innp{\nabla_{\vu} \Big(\frac{1}{2}\|\vu - \vu_0\|_p^2\Big)\Big|_{\vu = \vu_0}, \vu^* - \vu_0}\\
    &= \frac{1}{2}\|\vu^* - \vu_0\|_p^2.
\end{align*}

\noindent\textbf{Proof of Part (ii).} In this case, $q = p$, $\phi_p(\vu, \vv) = \frac{1}{p}\|\vu - \vv\|_p^p,$ and $m_p = 1.$  Using Proposition~\ref{prop:step-to-grad}, $\|\vu_k - \vub_k\|_p^p = \frac{{a_k}^{p^*}}{\beta^{p^*}}\| F(\vu_k)\|_{p^*}^{p^*}$ and $\|\vu_{k+1} - \vu_k\|_p^p = {a_k}^{p^*}\|F(\vub_k)\|_{p^*}^{p^*}.$ Combining with Eq.~\eqref{eq:egpp-det-h_k}, we have:
\begin{equation}\label{eq:h_k-p>2}
    \begin{aligned}
        0 \leq \;&\frac{1}{p}\|\vu^* - \vu_k\|_p^p - \frac{1}{p}\|\vu^* - \vu_{k+1}\|_p^p + \frac{(\beta - 1){a_k}^{p^*}}{p}\|F(\vub_k)\|_{p^*}^{p^*}\\ 
        &+ \frac{(a_k \Lambda_k \gamma - \beta){a_k}^{p^*}}{{p}\beta^{p^*}}\|F(\vu_k)\|_{p^*}^{p^*} + \frac{a_k \Lambda_k/\gamma-\beta}{p}\|\vub_k - \vu_{k+1}\|_p^p + a_k \delta_k.
    \end{aligned}
\end{equation}
Now let $\gamma = 1,$ $\beta = \frac{1}{2},$ $\delta_k = \delta > 0$, and $a_k = \frac{1}{2\Lambda_k} = \frac{1}{2\Lambda} = a$. Then $a_k \Lambda_k\gamma-\beta = a_k \Lambda_k/\gamma-\beta = 0$ and Eq.~\eqref{eq:h_k-p>2} simplifies to:
\begin{equation}\notag
    \frac{{a}^{p^*}}{2p}\|F(\vub_k)\|_{p^*}^{p^*} \leq \frac{1}{p}\|\vu^* - \vu_k\|_p^p - \frac{1}{p}\|\vu^* - \vu_{k+1}\|_p^p + a{\delta}.
\end{equation}
 Telescoping the last inequality and then dividing by $\frac{{a_k}^{p^*}(k+1)}{2p},$ we have:
\begin{equation}\label{eq:p>2-op-bnd}
    \frac{1}{k+1}\sum_{i=0}^k \|F(\vub_i)\|_{p^*}^{p^*} \leq \frac{2\|\vu^* - \vu_0\|_p^p}{a^{p^*}(k+1)} + \frac{2p \delta}{a^{p^* - 1}}. 
\end{equation}
Now, for \egpp~to be able to output a point $\vu$ with $\|F(\vu)\|_{p^*} \leq \epsilon,$ it suffices to show that for some choice of $\delta$ and $k$ we can make the right-hand side of Eq.~\eqref{eq:p>2-op-bnd} at most $\epsilon^{p^*}$. This is true because then \egpp~can output the point $\vub_i = \argmin_{0\leq i \leq k}\|F(\vub_i)\|_{p^*}.$ For stochastic setups, the guarantee would be in expectation, and \egpp~could output a point $\vub_i$ with $i$ chosen uniformly at random from $\{0,\dots, k\}$, as discussed in the proof of Theorem~\ref{thm:convergence-of-eg+}. 

Observe first that, as $\Lambda = \big(\frac{p - 2}{p\delta}\big)^{\frac{p-2}{2}}L^{p/2}$ and $p^* = \frac{p}{p-1},$ we have that:
\begin{align*}
    \frac{\delta}{{a}^{p^* - 1}} &= \delta (2\Lambda)^{p^* -1} = \delta 2^{\frac{1}{p-1}}\Lambda^{\frac{1}{p-1}}\\
    &= 2^{\frac{1}{p-1}}\delta^{\frac{p}{2(p-1)}}\Big(\frac{p-2}{p}\Big)^{\frac{p-2}{2(p-1)}}L^{\frac{p}{2(p-1)}}.
\end{align*}
Setting $\frac{2p \delta}{{a}^{p^* - 1}} \leq \frac{\epsilon^{p^*}}{2}$, recalling that $p^* = \frac{p}{p-1}$, and rearranging, we have:
$$
    \delta^{\frac{p^*}{2}} \leq \frac{\epsilon^{p^*}}{2^{\frac{2p-1}{p}}p}\Big(\frac{p}{p-2}\Big)^{\frac{p-2}{2p}p^*} L^{-p^*/2}.
$$
Equivalently:
$$
    \delta \leq \frac{\epsilon^2}{L \cdot 2^{\frac{2(2p-1)}{p}}p^{\frac{2(p-1)}{p}} (\frac{p-2}{p})^{\frac{p-2}{p}}}.
$$
It can be verified numerically that  $(\frac{p-2}{p})^{\frac{p-2}{p}}$ is a constant between $\frac{1}{e}$ and $1,$ while it is clear that  $2^{\frac{2(2p-1)}{p}}p^{\frac{2(p-1)}{p}} = O(p^2)$ is a constant that only depends on $p$. Hence, it suffices to set $\delta = \frac{\epsilon^2}{C_p L},$ where $C_p = 2^{\frac{2(2p-1)}{p}}p^{\frac{2(p-1)}{p}}.$

It remains to bound the number of iterations $k$ so that $\frac{2\|\vu^* - \vu_0\|_p^p}{a^{p^*}(k+1)} \leq \frac{\epsilon^{p^*}}{2}.$ Equivalently, we need $k+1 \geq \frac{4\|\vu^* - \vu_0\|_p^p}{a^{p^*}\epsilon^{p^*}}$. Plugging $\delta = \frac{\epsilon^2}{C_p L}$ into the definition of $\Lambda,$ using that $p^* = \frac{p}{p-1},$ and simplifying, we have:
\begin{align*}
    a^{p^*} &= (2\Lambda)^{p^*} = 2^{\frac{p}{p-1}} \Big(\frac{p-2}{p\delta}\Big)^{\frac{p-2}{2}\cdot\frac{p}{p-1}}L^{\frac{p}{2}\cdot\frac{p}{p-1}}\\
    &= O_p\bigg(\Big(\frac{1}{\epsilon}\Big)^{\frac{p(p-2)}{p-1}}L^p\bigg).
\end{align*}
Thus, 
\begin{align*}
    k = O_p\bigg(\Big(\frac{1}{\epsilon}\Big)^{\frac{p(p-2)}{p-1} + \frac{p}{p-1}}L^p\|\vu^* - \vu_0\|_p^p\bigg) = O_p\bigg(\Big(\frac{L\|\vu^* - \vu_0\|_p}{\epsilon}\Big)^p\bigg),
\end{align*}
as claimed.
\end{proof}
\fi

\begin{remark}
There are significant technical obstacles in generalizing the results from Theorem~\ref{thm:egpp_deterministic} to settings with $\rho > 0.$ In particular, when $p \in (1, 2),$ the proof fails because we take $\phi_p(\vu^*, \vu)$ to be the Bregman divergence of $\|\cdot - \vu_0\|_p^2$, and relating $\|\vub_k - \vu_k\|_p$ to $\|F(\vu_k)\|_{p^*}$ would require $\|\cdot\|_p^2$ to be smooth, which is not true. If we had, instead, used $\|\vu^* - \vu\|_p^2$ in place of $\phi_p(\vu^*, \vu)$, we would have incurred $\frac{1}{2}\|\vu^* - \vu_k\|_p^2 - \frac{m_p}{2}\|\vu^* - \vu_{k+1}\|_p^2$ in the upper bound on $h_k,$ which would not telescope, as in this case $m_p < 1.$ In the case of $p > 2,$ the challenges come from a delicate relationship between the step sizes $a_k$ and error terms $\delta_k.$ It turns out that it is possible to  guarantee local convergence (in the region where $\|F(\vub_k)\|_2$ is bounded by a constant less than 1) with $\rho > 0$, but $\rho$ would need to scale with $\mathrm{poly}(\epsilon)$ in this case.  As this is a weak result whose usefulness is unclear, we have omitted it.
\end{remark}

\paragraph{Stochastic oracle access.}
\ifaistats
\else
To obtain results for stochastic oracle access to $F,$ we only need to bound the terms $\mathcal{E}^s \defeq -a_k \innp{\vetab_k, \vub_k - \vu^*} - a_k\innp{\vetab_k - \veta_k, \vub_k - \vu_{k+1}}$ from Lemma~\ref{lemma:bound-on-hk-egpp} corresponding to the stochastic error in expectation, while for the rest of the analysis we can appeal to the results for the deterministic oracle access to $F.$ In the case of $p=2$, there is one additional term that appears in $h_k$ due to replacing $F(\vub_k)$ with $\tF(\vub_k).$ This term is simply equal to:
\begin{equation}
\begin{aligned}
    \frac{a_k\rho}{2}\ee[\|\tF(\vub_k)\|_2^2 - \|F(\vub_k)\|_2^2| \cfb_k]= \frac{a_k\rho}{2}\ee[\|F(\vub_k) + \vetab_k\|_2^2 - \|F(\vub_k)\|_2^2| \cfb_k] = \frac{a_k\rho}{2} {\bar{\sigma}_k}^2.
\end{aligned}
\end{equation}

We start by bounding the stochastic error $\mathcal{E}^s$ in expectation.
\fi

\begin{restatable}{lemma}{lemmastocherr}\label{lemma:stoch-err}
Let $\mathcal{E}^s = -a_k \innp{\vetab_k, \vub_k - \vu^*} - a_k\innp{\vetab_k - \veta_k, \vub_k - \vu_{k+1}}$, where $\vetab_k$ and $\veta_k$ are defined as in Eq.~\eqref{eq:noise-vectors} and all the assumptions of Theorem~\ref{thm:egpp-stochastic} below 
apply. Then, for $q$ defined by Eq.~\eqref{eq:q-def} and any $\tau>0$:
$$
    \ee[\mathcal{E}^s] \leq \frac{2^{q^*/2}{a_k}^{q^*}({\sigma_k}^2 + {\bar{\sigma}_k}^2)^{q^*/2}}{q^*\tau^{q^*}} + \ee\Big[\frac{\tau^q}{q}\| \vub_k - \vu_{k+1}\|_p^q\Big],
$$
where the expectation is w.r.t.~all the randomness in the algorithm.
\end{restatable}
\ifaistats\else
\begin{proof}
Let us start by bounding $-a_k \innp{\vetab_k, \vub_k - \vu^*}$ first. Conditioning on $\cfb_k$, $\vetab_k$ is independent of $\vub_k$ and $\vu^*$, and, thus: 
\begin{align*}
    \ee[-a_k \innp{\vetab_k, \vub_k - \vu^*}] = \ee\big[ \ee[-a_k \innp{\vetab_k, \vub_k - \vu^*}|\cfb_k] \big] = 0.
\end{align*}
The second term, $- a_k\innp{\vetab_k - \veta_k, \vub_k - \vu_{k+1}},$ can be bounded using H\"{older}'s inequality and Young's inequality as follows:
\begin{align*}
    \ee\big[- a_k\innp{\vetab_k - \veta_k, \vub_k - \vu_{k+1}}\big] &\leq \ee\big[a_k\|\vetab_k - \veta_k\|_{p^*}\| \vub_k - \vu_{k+1}\|_p\big]\\
    &\leq \ee\Big[\frac{{a_k}^{q^*}\|\vetab_k - \veta_k\|_{p^*}^{q^*}}{q^*\tau^{q^*}}\Big] + \ee\Big[\frac{\tau^q}{q}\| \vub_k - \vu_{k+1}\|_p^q\Big].
\end{align*}
It remains to bound $\ee\big[\|\vetab_k - \veta_k\|_{p^*}^{q^*}\big].$ Using the triangle inequality, 
\begin{align*}
    \ee\big[\|\vetab_k - \veta_k\|_{p^*}^{q^*}\big] &\leq \ee\Big[\big(\|\vetab_k\|_{p^*} + \|\veta_k\|_{p^*}\big)^{q^*}\Big]\\
    &= \ee\Big[\big(\big(\|\vetab_k\|_{p^*} + \|\veta_k\|_{p^*}\big)^{2}\big)^{q^*/2}\Big]\\
    &\leq \Big(\ee\big[\big(\|\vetab_k\|_{p^*} + \|\veta_k\|_{p^*}\big)^{2}\big]\Big)^{q^*/2},
\end{align*}
where the last line is by Jensen's inequality, as $q^* \in (1, 2],$ and so $(\cdot)^{q^*/2}$ is concave. Using Young's inequality and linearity of expectation:
\begin{align*}
    \ee\big[\big(\|\vetab_k\|_{p^*} + \|\veta_k\|_{p^*}\big)^{2}\big] &\leq 2 \Big(\ee\big[\|\vetab_k\|_{p^*}^2\big] + \ee\big[\|\veta_k\|_{p^*}^2\big]\Big)\\
    &\leq 2({\sigma_k}^2 + {\bar{\sigma}_k}^2).
\end{align*}
Putting everything together:
\begin{align*}
    \ee\big[\|\vetab_k - \veta_k\|_{p^*}^{q^*}\big] \leq 2^{q^*/2}({\sigma_k}^2 + {\bar{\sigma}_k}^2)^{q^*/2} 
\end{align*}
and
\begin{align*}
\ee[\mathcal{E}^s] &= \ee\big[- a_k\innp{\vetab_k - \veta_k, \vub_k - \vu_{k+1}}\big]\\
&\leq \frac{2^{q^*/2}{a_k}^{q^*}({\sigma_k}^2 + {\bar{\sigma}_k}^2)^{q^*/2}}{q^*\tau^{q^*}} + \ee\Big[\frac{\tau^q}{q}\| \vub_k - \vu_{k+1}\|_p^q\Big],
\end{align*}
as claimed.
\end{proof}

We are now ready to bound the total oracle complexity of \egpp~(and its special case \egp), as follows.
\fi

\begin{restatable}{theorem}{thmegppstoch}\label{thm:egpp-stochastic}
Let $p > 1$ and let $F:\rr^d\to\rr^d$ be an arbitrary $L$-Lipschitz operator w.r.t.~$\|\cdot\|_p$ that satisfies Assumption~\ref{assmpt:cohypo} for some $\vu^* \in \cu^*$. Given an arbitrary initial point $\vu_0 \in \rr^d,$ let the sequences of points $\{\vu_i\}_{i\geq 1}$, $\{\vub_i\}_{i\geq 0}$ evolve according to \eqref{eq:egp+} for some $\beta \in (0, 1]$ and positive step sizes $\{a_i\}_{i\geq 0}.$ Let the variance of a single query to the stochastic oracle $\tF$ be bounded by some $\sigma^2 < \infty.$ 
\begin{itemize}
    \item[(i)] Let $p = 2$ and $\rho \in \big[0, \bar{\rho}\big),$ where $\bar{\rho} = \frac{1}{4\sqrt{2}L}.$ If $\beta = \frac{1}{2}$ and $a_k = \frac{1}{2\sqrt{2}L},$ then \egpp~can output a point $\vu$ with $\ee[\|\tF(\vu)\|_2] \leq \epsilon$ with at most
    $$
        O\Big(\frac{L\|\vu^* - \vu_0\|_2^2}{\epsilon^2(\bar{\rho} - \rho)} \Big(1 + \frac{\sigma^2}{L\epsilon^2(\bar{\rho} - \rho)}\Big)\Big)
    $$
    oracle queries to $\tF.$ 
    \item[(ii)] Let $p \in (1, 2]$ and $\rho = 0.$ If $a_k = \frac{{m_p}^{3/2}}{2L}$ and $\beta = m_p$, then \egpp~can output a point $\vu$ with $\ee[\|\tF(\vu)\|_{p^*}] \leq \epsilon$ with at most
    $$
        O\Big(\frac{L^2\|\vu^* - \vu_0\|_p^2}{{m_p}^2\epsilon^2}\Big(1 + \frac{\sigma^2}{m_p \epsilon^2}\Big)\Big)
    $$
    oracle queries to $\tF,$ where $m_p = p-1.$ 
    \item[(iii)] Let $p > 2$ and $\rho = 0.$ If $\beta = \frac{1}{2}$ and $a_k = a = \frac{1}{4\Lambda},$ then \egpp~can output a point $\vu$ with $\ee[\|\tF(\vu)\|_{p^*}] \leq \epsilon$ with at most
    $$
        O_p\bigg(\Big(\frac{L\|\vu^* - \vu_0\|_p}{\epsilon}\Big)^p\Big(1 + \Big(\frac{\sigma}{\epsilon}\Big)^{p^*}\Big)\bigg)
    $$
    oracle queries to $\tF,$ where $p^* = \frac{p}{p-1}.$
\end{itemize}
\end{restatable}
\ifaistats
\else
\begin{proof}
Combining Lemmas~\ref{lemma:bound-on-hk-egpp} and \ref{lemma:stoch-err}, we have, $\forall k \geq 0$:
\begin{equation}\label{eq:hk-egpp-stoch}
    \begin{aligned}
       0 \leq \ee[ h_k] \leq&\; \frac{2^{q^*/2}{a_k}^{q^*}({\sigma_k}^2 + {\bar{\sigma}_k}^2)^{q^*/2}}{q^*\tau^{q^*}} + \ee\Big[\frac{\tau^q}{q}\| \vub_k - \vu_{k+1}\|_p^q\Big] + \frac{a_k \rho \bar{\sigma_k}^2}{2} + \ee\Big[\frac{a_k \rho}{2}\|\tF(\vub_k)\|_{p^*}^2\Big] \\
     &+ \ee\Big[\phi_p(\vu^*, \vu_k) - \phi_p(\vu^*, \vu_{k+1}) +  \frac{\beta - m_p}{q}\|\vu_{k+1} - \vu_k\|_p^q\Big]\\
     &+ \ee\Big[\frac{a_k \Lambda_k \gamma - \beta}{q}\|\vub_k - \vu_k\|_{p}^q + \frac{a_k \Lambda_k/\gamma - \beta m_p}{q}\|\vub_k - \vu_{k+1}\|_p^q + a_k \delta_k\Big],
    \end{aligned}
\end{equation}

\noindent\textbf{Proof of Part (i).} In this case, $q = 2$, $m_p = 1,$ $\delta = 0,$ $\Lambda_k = L,$ and $\phi_p(\vu^*, \vu) = \frac{1}{2}\|\vu^* - \vu\|_2^2,$ and, further, $\vu_{k+1} - \vu_k = -a_k F(\vub_k),$ so Eq.~\eqref{eq:hk-egpp-stoch} simplifies to
\begin{equation}\notag
    \begin{aligned}
       0 \leq \ee[ h_k] \leq&\; \frac{2{a_k}^2({\bar{\sigma_k}}^2 + {\bar{\sigma}_k}^2)}{2\tau^2}  + \frac{a_k \rho {\sigma_k}^2}{2} \\
     &+ \ee\Big[\frac{1}{2}\|\vu^* - \vu_k\|_2^2 - \frac{1}{2}\|\vu^* - \vu_{k+1}\|_2^2 +  \frac{{a_k}^2(\beta - 1) + a_k\rho}{2}\|\tF(\vub_k)\|_2^2\Big]\\
     &+ \ee\Big[\frac{a_k L \gamma - \beta}{2}\|\vub_k - \vu_k\|_{2}^2 + \frac{a_k L/\gamma - \beta + \tau^2}{2}\|\vub_k - \vu_{k+1}\|_2^2\Big].
    \end{aligned}
\end{equation}
Taking $\beta = \frac{1}{2},$ $\tau^2 = \frac{1}{4},$ $\gamma = \sqrt{2},$ and $a_k = \frac{1}{2\sqrt{2}L},$ and recalling that $\bar{\rho} = \frac{1}{4\sqrt{2}L},$ we have:
\begin{equation}\notag
    a_k (\bar{\rho} - \rho)\ee\big[\|\tF(\vub_k)\|_2^2\big] \leq \ee\big[\|\vu^* - \vu_k\|_2^2 - \|\vu^* - \vu_{k+1}\|_2^2\big] + 4{{a_k}^2({\sigma_k}^2 + {\bar{\sigma}_k)}^2}  + \frac{a_k \rho \bar{\sigma_k}^2}{2}.
\end{equation}
Telescoping the last inequality and dividing both sides by $a_k (\bar{\rho} - \rho)(k+1),$ we get:
\begin{equation}\notag
    \frac{1}{k+1}\sum_{i=0}^k \ee\big[\|\tF(\vub_i)\|_2^2\big] \leq \frac{2\sqrt{2}L\|\vu^* - \vu_0\|_2^2}{(k+1)(\bar{\rho} - \rho)} + \frac{\sqrt{2}\sum_{i=0}^k ({\sigma_i}^2 + {\bar{\sigma}_i}^2)}{L(\bar{\rho} - \rho)(k+1)} + \frac{\rho \sum_{i=0}^k {\bar{\sigma}_i}^2}{2(k+1)(\bar{\rho} - \rho)}.
\end{equation}
In particular, if the variance of a single sample of $\tF$ evaluated at an arbitrary point is $\sigma^2$ and we take $n$ samples of $\tF$ in each iteration, then:
\begin{equation}\notag
    \frac{1}{k+1}\sum_{i=0}^k \ee\big[\|\tF(\vub_i)\|_2^2\big] \leq \frac{2\sqrt{2}L\|\vu^* - \vu_0\|_2^2}{(k+1)(\bar{\rho} - \rho)} + \frac{\sigma^2(4\sqrt{2}/L + \rho)}{2n(\bar{\rho} - \rho)}.
\end{equation}
To finish the proof of this part, we require that both terms on the right-hand side of the last inequality are bounded by $\frac{\epsilon^2}{2}.$ For the first term, this leads to:
$$
    k = \left\lceil \frac{4\sqrt{2}L\|\vu^* - \vu_0\|_2^2}{\epsilon^2(\bar{\rho} - \rho)} - 1\right\rceil = O\Big(\frac{L\|\vu^* - \vu_0\|_2^2}{\epsilon^2(\bar{\rho} - \rho)}\Big).
$$
For the second term, the bound is:
$$
    n = \left\lceil  \frac{2\sigma^2(4\sqrt{2}/L + \rho)}{\epsilon^2(\bar{\rho} - \rho)} \right \rceil = O\Big(\frac{\sigma^2}{L\epsilon^2(\bar{\rho} - \rho)}\Big).
$$
Thus, the total number of required oracle queries to $\tF$ is bounded by:
$$
    k(1 + n) = O\Big(\frac{L\|\vu^* - \vu_0\|_2^2}{\epsilon^2(\bar{\rho} - \rho)} \Big(1 + \frac{\sigma^2}{L\epsilon^2(\bar{\rho} - \rho)}\Big)\Big).
$$
As discussed before, $\vub_i$ with $i$ chosen uniformly at random from $\{0, \dots, k\}$ will satisfy $\|\tF(\vub_i)\|_2 \leq \epsilon$ in expectation.

\noindent\textbf{Proof of Part (ii).} In this case, $q = 2$, $m_p = p-1,$ $\delta = 0,$ $\Lambda_k = L,$ and $\rho = 0.$ Thus, Eq.~\eqref{eq:hk-egpp-stoch} simplifies to:
\begin{equation}\notag
    \begin{aligned}
       0 \leq \ee[ h_k] \leq&\; \frac{2{a_k}^2({\sigma_k}^2 + {\bar{\sigma}_k}^2)}{2\tau^2} \\
     &+ \ee\Big[\phi_p(\vu^*, \vu_k) - \phi_p(\vu^*, \vu_{k+1}) +  \frac{\beta - m_p}{2}\|\vu_{k+1} - \vu_k\|_p^2\Big]\\
     &+ \ee\Big[\frac{a_k L \gamma - \beta}{2}\|\vub_k - \vu_k\|_{p}^2 + \frac{a_k L/\gamma - \beta m_p + \tau^2}{2}\|\vub_k - \vu_{k+1}\|_p^2\Big].
    \end{aligned}
\end{equation}
In this case, the same choices for $a_k$ and $\beta$ as in the deterministic case suffice. In particular, let $a_k = \frac{{m_p}^{3/2}}{2L},$ $\beta = m_p$, $\gamma = \frac{1}{\sqrt{m_p}},$ and $\tau^2 = \frac{{m_p}^2}{2}$. Then, using the fact that $\frac{1}{2}\|\vub_k - \vu_k\|_p^2 = \frac{{a_k}^2}{2\beta^2}\|\tF(\vu_k)\|_{p^*}^2,$ from Proposition~\ref{prop:step-to-grad}, we have
\begin{equation}\notag
    \frac{{a_k}^2m_p}{4\beta^2}\ee\big[\|\tF(\vu_k)\|_{p^*}^2\big] \leq \ee\Big[\phi_p(\vu^*, \vu_k) - \phi_p(\vu^*, \vu_{k+1})\Big] + \frac{{a_k}^2({\sigma_k}^2 + {\bar{\sigma}_k}^2)}{\tau}.
\end{equation}
Telescoping the last inequality and dividing both sides by $(k+1)\frac{{a_k}^2m_p}{4\beta^2},$ we have:
\begin{equation}\label{eq:p<2-stoch}
    \frac{1}{k+1}\sum_{i=0}^k \ee\big[\|\tF(\vu_i)\|_{p^*}^2\big] \leq \frac{16 L^2\phi_p(\vu^*, \vu_0)}{(k+1){m_p}^2} + \frac{8\sum_{i=0}^k ({\sigma_i}^2 + {\bar{\sigma_i}}^2)}{(k+1)m_p}.
\end{equation}
Now let ${\sigma_i}^2 = {\bar{\sigma}_i}^2 = \sigma^2/n$, where $\sigma^2$ is the variance of a single sample of $\tF$ and $n$ is the number of samples taken per iteration. Then, similarly as in Part (i), to bound the total number of samples it suffices to bound each term on the right-hand side of Eq.~\eqref{eq:p<2-stoch} by $\frac{\epsilon^2}{2}.$ The first term was already bounded in Theorem~\ref{thm:egpp_deterministic}, and thus we obtain:
$$
    k = O\Big(\frac{L^2\|\vu^* - \vu_0\|_p^2}{{m_p}^2\epsilon^2}\Big).
$$
For the second term, it suffices that:
$$
    n = O\Big(\frac{\sigma^2}{m_p \epsilon^2}\Big),
$$
and the bound on the total number of samples follows.

\noindent\textbf{Proof of Part (iii).} In this case, $q = p,$ $m_p = 1,$ $\rho = 0,$ $\phi_p(\vu^*, \vu) = \frac{1}{p}\|\vu^* - \vu\|_p^p$, and we take $\delta_k = \delta > 0,$ $\Lambda_k = \Lambda = \big(\frac{p-2}{p\delta}\big)^{\frac{p-2}{2}}L^{\frac{p}{2}}.$ Eq.~\eqref{eq:hk-egpp-stoch} now simplifies to:
\begin{equation}\label{eq:hk-egpp-stoch-p>1}
    \begin{aligned}
       0 \leq \ee[ h_k] \leq&\; \frac{2^{p^*/2}{a_k}^{p^*}({\sigma_k}^2 + {\bar{\sigma}_k}^2)^{p^*/2}}{p^*\tau^{p^*}}  \\
     &+ \ee\Big[\frac{1}{p}\|\vu^* - \vu_k\|_p^p - \frac{1}{p}\|\vu^* - \vu_{k+1}\|_p^p +  \frac{\beta - 1}{p}\|\vu_{k+1} - \vu_k\|_p^p\Big]\\
     &+ \ee\Big[\frac{a_k \Lambda \gamma - \beta}{p}\|\vub_k - \vu_k\|_{p}^p + \frac{a_k \Lambda/\gamma + \tau^p - \beta }{p}\|\vub_k - \vu_{k+1}\|_p^p + a_k \delta\Big].
    \end{aligned}
\end{equation}
Recall that, by Proposition~\ref{prop:step-to-grad}, $\frac{1}{p}\|\vu_{k+1} - \vu_k\|_p^p = \frac{{a}^{p^*}}{p}\|\tF(\vub_k)\|_{p^*}^{p^*}.$ Let $\beta = \frac{1}{2},$ $a_k = a = \frac{1}{4\Lambda},$ $\tau^p = \frac{1}{4},$ and $\gamma = 1.$ Then $\beta - 1 = - \frac{1}{2},$ $a_k\Lambda \gamma - \beta = - \frac{1}{4}<0,$ and $a_k \Lambda/\gamma + \tau^p - \beta = 0,$ and Eq.~\eqref{eq:hk-egpp-stoch-p>1} leads to:
\begin{equation}\notag
    \frac{a^{p^*}}{2p}\ee\big[\|\tF(\vub_k)\|_{p^*}^{p^*}\big] \leq \ee\Big[\frac{1}{p}\|\vu^* - \vu_k\|_p^p - \frac{1}{p}\|\vu^* - \vu_{k+1}\|_p^p\Big] + \frac{2^{\frac{4+p}{2(p-1)}}{a}^{p^*}({\sigma_k}^2 + {\bar{\sigma}_k}^2)^{p^*/2}}{p^*} + a \delta.
\end{equation}
Telescoping the last inequality and then dividing both sides by $\frac{a^{p^*}}{2p}(k+1),$ we have:
\begin{equation}\notag
     \frac{1}{k+1}\sum_{i=0}^k \ee\big[\|\tF(\vub_i)\|_{p^*}^{p^*}\big] \leq \frac{2\|\vu^* - \vu_0\|_p^p}{a^{p^*}(k+1)} + \frac{2^{\frac{3p+2}{2(p-1)}}p \sum_{i=0}^k({\sigma_i}^2 + {\bar{\sigma}_i}^2)^{p^*/2}}{p^*(k+1)} + \frac{2p\delta}{{a}^{p^*} - 1}.
\end{equation}
Now let $\sigma^2$ be the variance of a single sample of $\tF$ and suppose that in each iteration we take $n$ samples to estimate $F(\vub_i)$ and $F(\vu_i).$ Then ${\sigma_i}^2 = {\bar{\sigma_i}}^2 = \frac{\sigma^2}{n},$ and the last equation simplifies to
\begin{equation}\notag
     \frac{1}{k+1}\sum_{i=0}^k \ee\big[\|\tF(\vub_i)\|_{p^*}^{p^*}\big] \leq \frac{2\|\vu^* - \vu_0\|_p^p}{a^{p^*}(k+1)} + \frac{2^{\frac{p+2}{p-1}}p \sigma^{p^*}}{p^*n} + \frac{2p\delta}{{a}^{p^*} - 1}.
\end{equation}
To complete the proof, as before it suffices to show that we can choose $k$ and $n$ so that $\frac{2p\|\vu^* - \vu_0\|_p^p}{a^{p^*}(k+1)} + \frac{2p\delta}{{a}^{p^*} - 1} \leq \frac{\epsilon^{p^*}}{2}$ and $\frac{2^{\frac{p+2}{p-1}}p \sigma^{p^*}}{p^*n} \leq \frac{\epsilon^{p^*}}{2}$. For the former, following the same argument as in the proof of Theorem~\ref{thm:egpp_deterministic}, Part (ii), it suffices to choose $\delta = O_p(\frac{\epsilon^2}{L})$, which leads to:
$$
    k = O_p\bigg(\Big(\frac{L\|\vu^* - \vu_0\|_p}{\epsilon}\Big)^p\bigg).
$$
For the latter, it suffices to choose:
$$
    n = \frac{2^{\frac{p+2}{p-1} + 1}p \sigma^{p^*}}{p^*\epsilon^{p^*}} = O\bigg(\frac{p \sigma^{p^*}}{\epsilon^{p^*}}\bigg). 
$$
The total number of queries to the stochastic oracle is then bounded by $k(1+n).$
\end{proof}
\fi

\section{Discussion}

We introduced a new class of structured nonconvex-nonconcave min-max optimization problems and proposed a new generalization of the extragradient method that provably converges to a stationary point in Euclidean setups. Our algorithmic results guarantee that problems in this class contain at least one stationary point (an \eqref{eq:SVI} solution, see Remark~\ref{rem:SVI-existence}). The class we introduced generalizes other important classes of structured nonconvex-nonconcave problems, such as those in which an \eqref{eq:MVI} solution exists. We further generalized our results to stochastic setups and $\ell_p$-normed setups in which an \eqref{eq:MVI} solution exists. 
An interesting direction for future research is to understand to what extent we can further relax the assumptions about the structure of nonconvex-nonconcave problems, while maintaining computational feasibility of algorithms that can address them.

\section*{Acknowledgements}

We wish to thank Steve Wright for a useful discussion regarding convergence of sequences. We also wish to thank the Simons Institute for the Theory of Computing where some of this work was conducted.

JD was supported by the Office of the Vice
Chancellor for Research and Graduate Education at the University of Wisconsin–Madison with funding from the
Wisconsin Alumni Research Foundation and by the NSF Award  CCF-2007757. CD was supported by NSF Awards IIS-1741137, CCF-1617730, and CCF-1901292, by a Simons Investigator Award, by the Simons Collaboration on the Theory of Algorithmic Fairness, and by the DOE PhILMs project (No. DE-AC05-76RL01830). MJ was supported in part by the Mathematical Data Science program of the
Office of Naval Research under grant number N00014-18-1-2764.


\ifaistats
\else
\bibliographystyle{plainnat}
\fi
\bibliography{references,costisrefs}

\end{document}